\documentclass{article}

\usepackage{epic,eepic,epsfig,amssymb,amsmath,amsthm,graphics,enumerate}

\usepackage{multimedia}

\usepackage{pdftricks}

\begin{psinputs}
\usepackage{pstricks,pst-text}
\end{psinputs}

\usepackage{vmargin}
\usepackage{supertabular}
\usepackage{array}
\usepackage{multicol}
 
\setmarginsrb{3cm}{2cm}{3cm}{3cm}{75pt}{20pt}{20pt}{30mm}
\def\R{\mathbf R}
\def\bfS{\mathbf S}

\def\N{\mathbf N}

\def\0{\mathbf 0}

\newcommand\Dom{\mbox{Dom}\thinspace}
\newcommand\M{M}

\numberwithin{equation}{section}

\newtheorem{theorem}{Theorem}

\newtheorem{lemma}[theorem]{Lemma}
\newtheorem{proposition}[theorem]{Proposition}
\newtheorem{corollary}[theorem]{Corollary}
\newtheorem{definition}[theorem]{Definition}
\newtheorem{remark}[theorem]{Remark}
\newtheorem{example}[theorem]{Example}

\numberwithin{theorem}{section}

\newcommand{\dom}{\mathop{\rm Dom}}
\newcommand{\ran}{\mathop{\rm Ran}}

\newcommand{\graph}{\mathop{\rm Graph}}
\newcommand{\antigraph}{\mathop{\rm Antigraph}}

\makeindex

\begin{document}

\title{The intrinsic dynamics of optimal transport\thanks{
The authors thank the University of Nice Sophia-Antipolis, Berkeley's Mathematical Sciences
Research Institute (MSRI) and Toronto's Fields' Institute for the Mathematical Sciences
for their kind hospitality during various stages of this work. They
acknowledge partial support of RJM's research by Natural Sciences and Engineering Research
Council of Canada Grant 217006-08, and, during their Fall 2013 residency at MSRI, by the National Science
Foundation under Grant No. 0932078 000. 
\copyright 2015 by the authors.}}



\author{Robert J.~McCann\thanks{Department of Mathematics, University of Toronto, Toronto, Ontario, Canada M5S 2E4 {\tt mccann@math.toronto.edu}}
\and
Ludovic~Rifford\thanks{Universit\'e Nice Sophia Antipolis, Laboratoire \ J.A.\ Dieudonn\'e, UMR CNRS 7351, Parc  Valrose, 06108 Nice Cedex 02, France \& Institut Universitaire de France {\tt ludovic.rifford@math.cnrs.fr}}}




\maketitle

\begin{abstract}
The question of which costs admit unique optimizers in
the Monge-Kantorovich problem of optimal transportation between arbitrary probability densities is investigated.   
For smooth costs and densities on compact manifolds, the only known examples for which the optimal solution is always unique  require at least one of the two underlying spaces to be homeomorphic to
a sphere.  We introduce a (multivalued)
dynamics which the transportation cost induces between the target and source space, for which
the presence or absence of a sufficiently large set of periodic trajectories plays a role in determining whether or
not optimal transport is necessarily unique.  This insight allows us to construct smooth costs on a 
pair of compact manifolds with arbitrary topology,  
so that the optimal transportation between any pair of probility densities is unique. 
\end{abstract}

\section{Introduction}

Let $M$ and $N$ be smooth closed manifolds (meaning compact, without boundary) of dimensions $m$ and $n\geq 1$ respectively, and $c:M\times N \rightarrow \R$ a continuous cost function. Given two probability measures $\mu$ and  $\nu$ respectively on $M$ and $N$,  the Monge problem consists in minimizing the transportation cost 
\begin{eqnarray}\label{monge}
\int_{M\times N} c\bigl(x,T(x) \bigr) \, d\mu(x),
\end{eqnarray}
among all transport maps from $\mu$ to $\nu$, that is such that $T_{\sharp} \mu =\nu$. A classical way to prove existence and uniqueness of optimal transport maps is to relax the Monge problem into the Kantorovitch problem. That problem is a linear optimization problem under convex constraints, it consists in minimizing the transportation cost 
\begin{eqnarray}\label{kantorovitch}
\int_{M\times N} c(x,y) \, d\gamma(x,y),
\end{eqnarray}
among all transport plans between $\mu$ and $\nu$,  meaning $\gamma$ belongs to the set
$\Pi(\mu,\nu)$ of non-negative measures having marginals $\mu$ and $\nu$. By classical (weak) compactness arguments, minimizers for the Kantorovitch problem always exist. A way to get existence and uniqueness of minimizers for the Monge problems is to show that any minimizer of (\ref{kantorovitch}) is supported on a graph. Assuming that $c$ is Lipschitz and $\mu$ is absolutely continuous with respect to the Lebesgue measure,   a condition which guarantees this graph
property is the following 
nonsmooth version \cite{ChiapporiMcCannNesheim10} \cite{riffordmemoir} of the 
TWIST condition 
$$
D_x^- c \bigl(\cdot,y_1\bigr) \cap D_x^- c \bigl(\cdot,y_2\bigr) = \emptyset \qquad \forall y_1\neq y_2 \in N, \, \forall x\in M, 
$$
where $D_x^- c (\cdot,y_i)$ denotes the sub-differential of the function $x\mapsto c(x,y_i)$ at $x$.
In this case, it is well-known how to use linear programming duality to prove that the Kantorovich minimizer is unique,  and that Monge's infimum is attained  \cite{Gangbo95} \cite{Levin99}.

Examples of Lipschitz costs satisfying the nonsmooth TWIST are given by any cost coming from variational problems associated with Tonelli Lagrangians of class $C^{1,1}$ (see \cite{bb07}), like the square of Riemannian distances (see \cite{mccann01}). Those costs are never $C^1$ on compact manifolds such as $M\times N$.  As a matter of fact, any cost $c:M\times N\rightarrow \R$ of class $C^1$ admits a triple $x \in M,y_1 \in N, y_2 \in N,$ (take $y_1$ with $c(x,y_1) = \min \{c(x,\cdot)\}$ and $y_2$ with $c(x,y_2) = \max \{c(x,\cdot)\}$ ) such that  
$$
\frac{\partial c}{\partial x} \bigl( x, y_1\bigr) = \frac{\partial c}{\partial x} \bigl( x, y_2\bigr), 
$$ 
violating the nonsmooth TWIST condition. Indeed, we shall show the following holds.

\begin{theorem}[Non-genericity of twist]
\label{THMmotivation}
Let $c:M\times N \rightarrow [0,\infty)$ be a cost function of class $C^2$. Assume that $\dim M = \dim N$ and 
\begin{eqnarray}\label{ASSmotivation}
\exists (\bar{x},\bar{y})\in M\times N \quad \mbox{such that} \quad \frac{\partial^2 c}{\partial x \partial y}(\bar{x},\bar{y}) \quad \mbox{is invertible}. 
\end{eqnarray}
Then there is a pair  $\mu, \nu$ of probability measures respectively on $M$ and $N$ which are both absolutely continuous with respect to the Lebesgue measure for which there is a unique optimal transport plan for (\ref{kantorovitch})
 and such that this plan is not supported on a graph. The set of costs $c$ satisfying (\ref{ASSmotivation}) is open and dense in $ C^2(M\times N;\R)$. 
\end{theorem}


The conclusion of Theorem \ref{THMmotivation}  implies that solutions  for the Monge problem  with smooth cost do not generally exist
in a compact setting. The purpose of the present paper is to study sufficient conditions for uniqueness of the Kantorovitch optimizer, and to exhibit smooth costs 
on arbitrary manifolds 
for which optimal plans are unique, despite the fact that such plans are not generally concentrated on graphs.  
Some examples of such costs have been given in \cite{gm00} \cite{akm11} (see also \cite{bc09}). 
However,  if uniqueness is to hold for arbitrary absolutely continuous $\mu$ and $\nu$ on $M$ and $N$,
all previous examples which we are aware of that involve smooth costs have required at least one of the two compact manifolds 
be homeomorphic to a sphere.  
Here we go far beyond this,  to construct examples of such costs on compact manifolds
whose topology can be arbitrary.
Our main idea is to relate the uniqueness of the Kantorovitch optimizer to a multivalued dynamics induced by the cost which does not seem to have been considered previously.


Before stating our results, we need to introduce some definitions. 


Denoting the non-negative integers by $\N=\{0,1,2,\ldots\}$ and the positive integers by 
$\N^* =\N \setminus \{0\}$,
we begin recalling the well-known notion of $c$-cyclical monotonicity.

\begin{definition}[$c$-cyclical monotonicity]
A set $S \subset M \times N$ is $c$-cyclically monotone when for all $I \in \N^*$
and $(x_i,y_i) \in S$ for $i=1, \ldots, I$ with  $x_{I+1}=x_1$, we have 
$$
\sum_{i=1}^I \left[ c(x_{i+1},y_i) - c(x_i,y_i)\right] \geq 0. 
$$
\end{definition}

For given $\mu,\nu$ and $c$, 
it is also well-known \cite{gm96} that some closed $c$-cyclically monotone subset $S\subset M\times N$
contains the support of all optimizers to \eqref{kantorovitch}.
Note that of course, any subset of a  $c$-cyclically monotone set is  $c$-cyclically monotone as well. We come now to the concepts which will play a major role.

\begin{definition}[Alternant chains]
For each $(x,y) \in M \times N$ assume $c(x,\cdot)$ and $c(\cdot,y)$
are differentiable.
\label{Laltchain}
Fixing $S \subset M \times N$, we call {\em chain} in $S$ of length $L\ge 1$ (or {\em $L$-chain} for short) any ordered family of pairs 
$$
 \Bigl( (x_1,y_1), \ldots, (x_L,y_L)\Bigr) \in S^L
$$
such that the set
$$
\Bigl\{ (x_1,y_1), \ldots (x_L,y_L)\Bigr\}
$$
is $c$-cyclically monotone and for every $l=1, \ldots, L-1$ there holds, either 
\begin{equation}\label{altc1}
x_{l} = x_{l+1} \quad \mbox{and} \quad y_l \neq y_{l+1} = y_{\min\{L,l+2\}}\quad  \mbox{ and }  \quad \frac{\partial c}{\partial x} (x_l,y_l) = \frac{\partial c}{\partial x} (x_l,y_{l+1}), 
\end{equation}
or
\begin{equation}\label{altc2}
y_l = y_{l+1} \quad \mbox{and} \quad x_l \neq x_{l+1} = x_{\min\{L,l+2\}}  \quad  \mbox{ and } \quad \frac{\partial c}{\partial y} (x_l,y_l) = \frac{\partial c}{\partial y} (x_{l+1},y_{l}).
\end{equation}
The chain 
is called cyclic if its projections onto $M$ and $N$ each consist of $L/2$ distinct points,
in which case $L$ must be even with $y_L=y_1$ and $x_L \ne x_1$.
\end{definition}


Note the existence of any cyclic 
chain $\left((x_1,y_1), \ldots (x_L,y_L)\right)$ permits the construction of an infinite 
chain  $\{(x_l,y_l)\}_{l\in \N^*}$ by
\begin{equation}\label{periodic chain}
\bigl(x_{kL+l},y_{kL+l}\bigr) := \bigl(x_{l},y_{l}\bigr) \qquad \forall k\geq 1, \, \forall l\in \{1, \ldots, L\}.
\end{equation}

Our first result is the following:




\begin{theorem}[Optimal transport is unique if long chains are rare]
\label{THM1}
Fix a cost $c\in C^1(M\times N)$. Choose Borel probability measures 
$\mu$ on $M$ and $\nu$ on $N$, both 
absolutely continuous with respect to Lebesgue, 
and let $\Pi_0$ denote the set of all optimizers for \eqref{kantorovitch}  on $\Pi(\mu,\nu)$.
Let $E_0 \subset M \times N$ be a $\sigma$-compact set which is negligible
for all $\gamma \in \Pi_0$, and denote its complement by $\tilde {\mathcal S} := (M \times N) \setminus E_0$.
Let $E_\infty$ 
denote the set of points which occur in $k$-chains in $\tilde {\mathcal S}$ for arbitrarily large $k$.  Then $E_\infty$ 
and its projections  $\pi^M(E_\infty)$ and $\pi^N(E_\infty)$
are Borel.
If $\gamma( 
E_\infty)=0$
for every $\gamma \in \Pi_0$,  then $\Pi_0$ is a singleton. %
\end{theorem}

\begin{remark}[Extension to singular marginals]\label{R singular measures}
When $c \in C^{1,1}$,  we can relax the absolute continuity of $\mu$ and $\nu$ in the preceding
theorem provided neither
concentrates positive mass on a $c-c$ hypersurface. Here {\em $c-c$ hypersurface} refers to one which can be parameterized in local coordinates as the graph of a difference of convex functions \cite{Zajicek79} \cite{gm96} \cite{Gigli11}.
\end{remark}

\begin{corollary}[Sufficient notions of rarity]\label{R notions of rarity}
The condition $\gamma(
E_\infty)=0$ in the statement of the theorem, and therefore its conclusions, follow from either 
$\mu(\pi^M(
E_\infty))=0$ or $\nu(\pi^N(
E_\infty))=0$.
\end{corollary}


If there is a uniform bound $K$ on the length of all chains in $M \times N$,
then our theorem applies a fortiori with ${\mathcal S} = M \times N$ and $E_0=\emptyset$,
since $
E_\infty=\emptyset$. 
We shall see this occurs in many cases of interest, including for the smooth costs that
we construct on compact manifolds with arbitrary topology.  The bound $K$ will depend on the 
topology. 
On the other hand, an obstruction to the uniqueness of 
optimal plans is the existence of a non-negligible set of periodic orbits. 
As shown below, such a property is not typical:  it fails to occur for costs $c$ in a countable
intersection $\mathcal C$ of open dense sets.  Such a countable intersection is
called {\em residual}.


\begin{theorem}[Costs admitting cyclic chains are non-generic]\label{THM2}
When $\dim M=\dim N$, 
there is a residual set $\mathcal{C}$ in $C^{\infty}(M\times N;\R)$ such that no cost in $\mathcal{C}$ admits
cyclic chains,  and for every cost $c\in \mathcal{C}$, there is a nonempty closed set $\Sigma \subset M\times N$ of zero (Lebesgue) volume such that 
$$
 \frac{\partial^2 c}{\partial x \partial y}(x,y) \quad \mbox{is invertible for any } \, (x,y) \in M\times N \setminus \Sigma.
 $$
\end{theorem}



In the terminology of Hestir and Williams \cite{hw95}, the absence of cyclic chains is sufficient to define (formally)  a rooting set whose measurability would be sufficient for uniqueness. We refer the reader to Section \ref{SEClimb} for further details on their approach and its aftermath \cite{bc09} \cite{akm11} \cite{Moameni14p}. We do not know if uniqueness of optimal plans between absolutely continuous measures holds for generic costs.  However, elaborating on a celebrated result by Ma\~n\'e \cite{mane96} in the framework of Aubry-Mather theory, we are able to prove that  uniqueness of optimal transport plans holds for generic costs in $C^k$ if the marginals are fixed.   In $C^0$, such a result was known already to Levin \cite{Levin08}.

\begin{theorem}[Optimal transport between given marginals is generically unique]
\label{THM3}
Fix Borel probability mesures on compact manifolds $M$ and $N$. For each $k\in \N \cup \{\infty\}$,
there exists a residual set $\mathcal{C} \subset C^k(M\times N;\R)$ such that for every $c\in \mathcal{C}$,  there is a unique optimal plan between $\mu$ and $\nu$.
\end{theorem}

The paper is organized as follows. We provide examples of costs satisfying the above results in Section \ref{SECexamples}. We develop preliminaries on numbered limb systems and details on Hestir and Williams' rooting sets in Section \ref{SEClimb}. We give the proofs of Theorem \ref{THM1} in Section \ref{SECproofsTHM1}, of Theorem \ref{THM2} in Section \ref{SECproofTHM2}, and finally of Theorem \ref{THM3} in Section \ref{SECproofTHM3}.


\section{Examples and applications}\label{SECexamples}

\subsection{Quadratic cost on a strictly convex set}

Let us begin by recasting an example of Gangbo and McCann  \cite{gm00} into the framework
of (alternant) chains.


Fix $N \subset \R^{m+1}$. Let $\M$ be the boundary of a strictly convex body $ \Omega \subset \R^{m+1}$, 
that is a closed set which is the boundary of a bounded open convex set 
and such that for any $z,z'\in \M$,
$$
[z,z']  \subset M \quad \Longrightarrow \quad z=z',
$$
where $[z,z']$ is the segment joining $z$ to $z'$.
We aim to show that for any measures $\mu$ and $\nu$ ($\mu$ being absolutely continuous w.r.t. the Hausdorff $m$-dimensional measure $\mathcal{H}^{m}$ measure on $\M$), we have uniqueness of optimal plans for the cost 
$$
c(x,y) = \frac{1}{2} |x-y|^2 \qquad \forall (x,y) \in \M \times N.
$$
Let $\mathcal{P}(\M\times N)$ denote the Borel probability measure on $\M \times N$ and 
$\pi^M: \M \times N \rightarrow \M$ and $\pi^N: \M \times N \rightarrow N$
the projections onto the first and second variables. Let $\mu$ and $\nu$ be probability measures 
on $\M$ and $N$. We recall that the support 
$\Gamma \subset \M \times N$ of any  plan 
$\bar{\gamma} \in \mathcal{P}(\M \times N)$ minimizing 
\begin{eqnarray}\label{Kantorovitch inf}
\inf \left\{ \int_{\M \times N} c(x,y) d\gamma (x,y) \, \vert \, \gamma \in \Pi(\mu,\nu) \right\}
\end{eqnarray}
is $c$-cyclically monotone, which in the case $c(x,y)=|y-x|^2 /2$ reads  
$$
\sum_{i=1}^I \langle y_i, x_i - x_{i+1}\rangle \geq 0,
$$
 for all positive integer $I$, $i=1, \ldots, I$, $(x_i,y_i) \in A$, $x_{I+1}=x_1$. The uniqueness of optimal plans will follow easily from the next lemma.

\begin{lemma}[Interior links are never exposed]
\label{lemmahyperplane}
Fix a hypersurface
$M \subset \R^{m+1}$, possibly incomplete.
For any submanifold $N \subset \R^{m+1}$ of dimension $n\le m+1$, let
$c(x,y)$ denote the restriction of $\frac12 |x-y|^2$ to $M \times N$.
If $((x_0,y), (x_2,y),(x_2,y'),(x_4,y'))$ is a chain in $M \times N$,
then no hyperplane strictly separates $x_2$ from $M \setminus \{x_2\}$.
\end{lemma}

\begin{proof}
To derive a contradiction, suppose
$((x_0,y), (x_2,y),(x_2,y'),(x_4,y'))$ forms a chain in $M \times \R^{m+1}$, 
yet $x_2$ is strictly separated from $M \setminus \{x_2\}$
by a hyperplane with inward normal $n_2$, i.e. 
\begin{equation}\label{strict separation}
\langle x-x_2,n_2 \rangle > 0
\end{equation}
for all $x \in M \setminus \{x_2\}$.    The chain conditions imply
$y'-y = \alpha n_2$ for some $\alpha \in \R$.

On the other hand,  pairwise monotonicity of the points in the chain imply
\begin{eqnarray*}
\langle x_4-x_2,y'-y \rangle = & \alpha \langle x_4-x_2,n_2 \rangle & \ge 0
\\ \langle x_2-x_0,y'-y \rangle = & \alpha \langle x_2-x_0,n_2 \rangle & \ge 0.
\end{eqnarray*}
Using  \eqref{strict separation} we deduce $\alpha \ge 0$ from the first inequality
and $\alpha \le 0$ from the second. But $\alpha =0$ yields
$y'=y$,  contradicting the definition of a chain.
\end{proof}

As a consequence we have:

\begin{corollary}[Chain bounds for strictly convex hypersurfaces]
\label{corollaryconvex}
If $c$ is the restriction of $|x-y|^2$ to $M \times N$ as above, where $M \subset \R^{m+1}$
is a strictly convex hypersurface,
then $\M \times N$ contains no chain of length $ L \geq 5$.
Moreover,  the projection of any $4$-chain in $M\times N$ onto $N$ consists of three distinct points,
while its projection onto $M$ consists of two distinct points.
If $N \subset \R^{m+1}$ is also a strictly convex hypersurface, then
$M \times N$ contains no chain of length $L \ge 4$.
\end{corollary}



\begin{proof}
If a chain of length $L \ge 5$ exists, it begins either with 
\begin{equation}\label{4-chain}
((x_1,y_2),(x_3,y_2),(x_3,y_4),(x_5,y_4))
\end{equation}
or 
$((x_2,y_1),(x_2,y_3),(x_4,y_3),(x_4,y_5),(x_6,y_5))$.
Since $M$ is strictly convex, each point $x \in M$ is exposed, meaning it can be strictly separated 
from $M \setminus \{x\}$ by a hyperplane.   In the first case Lemma
\ref{lemmahyperplane} would be violated by the chain \eqref{4-chain} since $x_2$
is an exposed point of $M$;
in the second it would be violated by the chain $((x_2,y_3),(x_4,y_3),(x_4,y_5),(x_6,y_5))$
since $x_4$ is an exposed point of $M$.  We are forced to conclude that no chain of length
$L \ge 5$ can exist.
Moreover,  any chain of length $L=4$ in $M \times N$ must take the 
form $((x_2,y_1),(x_2,y_3),(x_4,y_3),(x_4,y_5))$ hence project onto three points $y_i \in N$.
The $y_i$ must all be distinct since $y_1 \ne y_3 \ne y_5$ from the definition of chain, while
$y_5=y_1$ would make the chain cylic,  in which
case it can be extended to an infinite chain \eqref{periodic chain} contradicting
non-existence of a chain of length $5$.  
The projection onto $M$ therefore consists of the two points $x_2 \ne x_4$,
which are distinct by the definition of chain.

If $N \subset \R^{m+1}$ is also a strictly convex hypersurface then by symmetry,
$M \times N$ can contain no chain which projects to more than two points on $M$
and two points on $N$,  hence no chain of length $L\ge 4$.
\end{proof}



\begin{example}\label{EXlake}
Let us consider the example of the lake that already appeared in \cite{gm00} and \cite{ChiapporiMcCannNesheim10}. Let $M=N$ be the unit circle in the plane, that is the circle centered at the origin of radius $1$ equipped with the quadratic cost $c(x,y)=|y-x|^2 /2$. Consider a small auxiliary circle centered on the vertical axis, for example the circle centered at $(0,-5/2)$ of radius $1/8$, denote by $\tilde{\psi}$ the distance function to the disc $D$ enclosed by the small circle (see Figure \ref{FIGlake}). By construction, $\tilde{\psi}$ is convex and differentiable at every point of $M$ with a gradient of norm 1. 

\begin{figure}[!h]
\label{FIGlake}
\begin{pdfpic} 
\psset{unit=0.75cm}
\begin{pspicture}(12,10)
\pscircle(6,7){2}
\psdots[dotsize=2pt](6,7)
\psdots[dotsize=3pt](5,5.27)
\pscircle[fillstyle=solid,fillcolor=gray](6,2){.35}
\psdots[dotsize=2pt](6,2)
\rput[c](7.5,5){$M$}
\rput[c](6,1.3){$D$}
\rput[c](4.8,5.07){$x$}
\psline[linewidth=0.5pt,linecolor=gray,linestyle=dashed](6,7)(5,5.27) 
\psline[linewidth=0.5pt,linecolor=gray](6,2)(5,5.27) 
\psline[linewidth=0.5pt,linecolor=gray](6,7)(5.415,8.912) 
\psdots[dotsize=3pt](5.415,8.912) 
\rput[r](5.215,9){$\bar{y}(x)$}
\psdots[dotsize=3pt](4.055,6.559)
\rput[r](3.855,6.359){$\hat{y}(x)$}
\psline[linewidth=0.5pt,linecolor=gray,linestyle=dashed](5.415,8.912)(4.055,6.559)
\psline[linewidth=0.25pt,linecolor=lightgray](0,1)(12,1)(12,10)(0,10)(0,1) 
\end{pspicture}
\end{pdfpic}
\caption{The lake}
\end{figure}
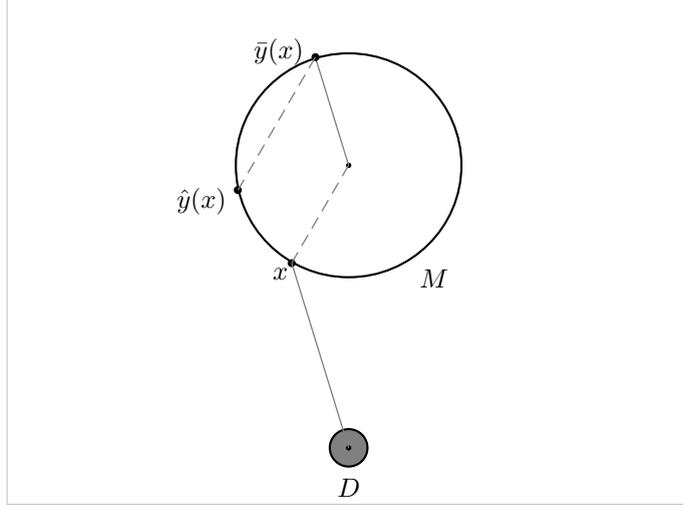

Then we set 
$$
\psi(x) := \tilde{\psi}(x) - \frac{1}{2} |x|^2 \qquad \forall x \in M
$$
and 
$$
\phi(y) := \min \Bigl\{ \psi(x)+c(x,y) \, \vert \, x \in M\Bigr\} \qquad \forall y \in M.
$$
By construction, we check that 
\begin{equation}\label{c-max}
\psi(x)= \max \Bigl\{ \phi(y)-c(x,y) \, \vert \, y \in M\Bigr\} \qquad \forall x \in M.
\end{equation}
Moreover, for every $x\in M$, the gradient $\bar{y}(x) := \nabla_x\tilde{\psi} \in \R^2$ belongs to 
the set $\partial_c\psi(x)\subset M$ of optimizers for \eqref{c-max}. 
As a matter of fact, we have by convexity of $\tilde{\psi}$,
$$
\tilde{\psi}(x') -\tilde{\psi}(x) \geq \langle \bar{y}(x), x'-x\rangle \qquad \forall x, x' \in \R^2.
$$
Which can be written as
$$
\psi(x) \leq \psi(x') + \frac{1}{2} \bigl| x'-\bar{y}(x)\bigr|^2 -\frac{1}{2}  \bigl| x-\bar{y}(x)\bigr|^2   \qquad \forall x, x'\in \R^2.
$$
Taking the minimum over $x'\in M$, we infer that ($\bar{y}(x)\in M$)
$$
c\bigl(x,\bar{y}(x)\bigr) \leq \phi\bigl( \bar{y}(x)\bigr) - \psi(x) \qquad \forall x \in M
$$
which, because $\phi-\psi\leq c$ implies that 
$$
c\bigl(x,\bar{y}(x)\bigr) = \phi\bigl( \bar{y}(x)\bigr) - \psi(x) \qquad \forall x \in M,
$$
which means that $\bar{y}(x)= \nabla_x\tilde{\psi}$ always belongs to $\partial_c\psi(x)$. For every $x\in M$, we set  
$$
\hat{y}(x) := \bar{y}(x) + \lambda(x) x \in M,
$$
where $\lambda(x) \geq 0$ is the largest nonnegative real number $\lambda$ such that $ \bar{y}(x) + \lambda x$ belongs to $M$ (in other terms, $\hat{y}(x)$ is the intersection of the open semi-line starting from $\bar{y}(x)$ with vector $x$ if the intersection is nonempty and $\hat{y}(x)=\bar{y}(x)$ otherwise). For every $x\in M$, the point $\hat{y}(x)$ belongs to $\partial_c\psi(x)$ as well. As a matter of fact, by convexity of $M$, the fact that the normal to $M$ at $x$ is $x$ itself and  the convexity of $\tilde{\psi}$, we have for every $x, x'\in M$,
$$
\langle \hat{y}(x), x'-x\rangle = \langle \bar{y}(x), x'-x\rangle + \lambda(x) \langle x,x'-x\rangle \leq \langle \bar{y}(x), x'-x\rangle \leq \tilde{\psi}(x') -\tilde{\psi}(x).
$$
Proceeding as above we infer that $\hat{y}(x)$ belongs to  $\partial_c\psi(x)$. We can check easily that for every point $x$ close to the south pole $(-1,0)$ the points $\hat{y}(x), \bar{y}(x)$ are distinct (see Figure \ref{FIGlake}). Proceeding as in the proof of Theorem \ref{THMmotivation}, we can construct an example of optimal transport plan which is not concentrated on a graph.
\end{example}

\subsection{Quadratic cost on nested strictly convex sets} 

Let 
\begin{equation}\label{nested convex}
\Omega_1 \subset \Omega_2 \subset \ldots \subset \Omega_L.
\end{equation}
be a nested family of strictly convex bodies with differentiable boundaries in $\R^{m+1}$.
Set $M= \cup_{i=1}^L U_i$, where $U_i = \partial \Omega_i \setminus 
\partial \Omega_{i-1}$
is an embedding of (a portion of) 
the unit sphere.

\begin{figure}[!h]
\label{FIGnested}
\centering
\begin{pdfpic}
\centering 
\psset{unit=0.75cm}
\begin{pspicture}(12,10)
\pscircle(6,5){1}
\psellipse(6.25,5.5)(1.5,2.25)
\pscircle(6.75,5.5){3}
\psccurve(4,2)(9,2)(11,8)(3,8)
\rput[c](6.75,6.05){$\mathcal{S}_1$}
\rput[c](7.4,3.5){$\mathcal{S}_2$}
\rput[c](8.95,8){$\mathcal{S}_3$}
\rput[c](9,1.5){$\mathcal{S}_4$}
\psline[linewidth=0.25pt,linecolor=lightgray](0,1)(12,1)(12,10)(0,10)(0,1) 
\end{pspicture}
\end{pdfpic}
\caption{Nested convex sets}
\end{figure}
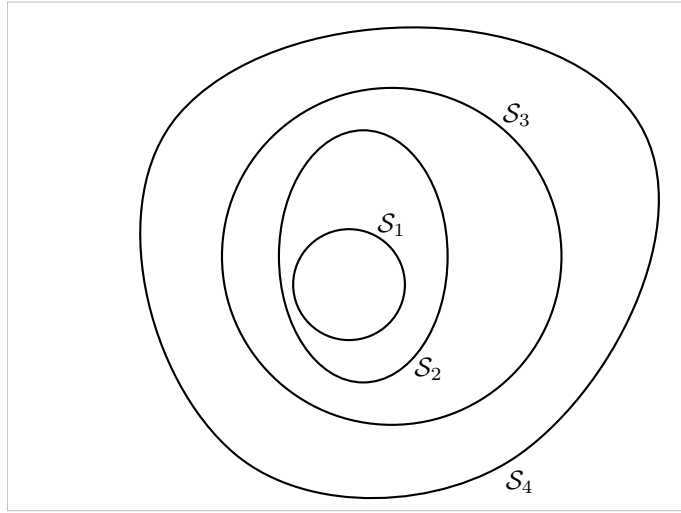


\begin{lemma}[Chain length bounds for nested strictly convex boundaries]
\label{lemmanested}
If $c(x,y)$ denotes the restriction of $\frac 12 |x-y|^2$ to $C^1$ manifolds 
$M_L,N\subset \R^{m+1}$,
and $\M_L:= \partial \Omega_1 \cup \ldots \cup \partial \Omega_L$ 
is a union
of boundaries of a nested sequence \eqref{nested convex}
of strictly convex bodies $\Omega_i \subset \R^{m+1}$, then $M_L \times N$ 
contains no chain of length $4L+1$.  Moreover,  any chain of length $4L$ 
has projections onto $M_L$ (respectively $N$) which consist of $2L$ (respectively $2L+1$) distinct points.
\end{lemma}

\begin{proof}
We prove the result by induction on $L$. Corollary \ref{corollaryconvex} gives the result for $L=1$. 
So assume that the property is proved for $L\ge1$ and prove it for $L+1$.  
Note that although $M_L$ may not be a submanifold of $\R^{m+1}$
(if the boudaries of $\Omega_i$ and $\Omega_{i+1}$ intersect),  it may be regarded
as $C^1$ embedding of the disjoint union $\bigcup_{i=1}^L U_i$ of potentially incomplete manifolds
$U_i = \partial \Omega_i \setminus \partial \Omega_{i-1}$. 


Any chain in $M_{L+1} \times N$ of length $4L+5$ takes one of the forms
\begin{eqnarray}
\label{first form}
((x_1, y_2),(x_3, y_2),(x_3,y_4), \ldots, (x_{4L+3},y_{4L+4}), (x_{4L+5},y_{4L+4}), (x_{4L+5},y_{4L+6}))
&& \mbox{or} \\
((x_2, y_1),(x_2, y_3),(x_4,y_3), \ldots, (x_{4L+4},y_{4L+3}), (x_{4L+4},y_{4L+5}), (x_{4L+6},y_{4L+5}))&.&
\label{second form}
\end{eqnarray}
Strict convexity of $\partial \Omega_{L+1}$ shows any $x \in \partial \Omega_{L+1}$ can be separated from $M_{L+1} \setminus \{x\}$ by a hyperplane.  Lemma \ref{lemmahyperplane} therefore implies 
$\{x_4, \ldots, x_{4L+3}\} \subset M_L$, so that apart from possibly the first and last pairs of points,  
the chains \eqref{first form}--\eqref{second form} above are contained in
$M_L \times N$. But this contradicts the inductive hypothesis,  which asserts that $M_L \times N$ contains no chain
of length $4L+1$.

Similarly,  if $M_{L+1} \times N$ contains a chain of length $2L+4$, it must take the form of the first
2L+4 points in \eqref{second form} rather than \eqref{first form};  in the latter case $\{x_3,\ldots, x_{4L+3}\} \subset M_L$
whence $M_L \times N$ would contain a chain of length $4L+3$, contradicting the inductive hypothesis.
In the former case,  $\{x_4, \ldots, x_{4L+2}\} \subset M_L$,   whence $M_L \times N$ contains a chain
of length $4L$ which the inductive hypothesis asserts 
is comprised of $2L$ distinct points $X_4^{4L+2} :=\{x_4,x_6,\ldots, x_{4L+2}\}$ 
and $2L+1$ distinct points $Y_3^{4L+3}:=\{y_3,y_5,\ldots,y_{4L+3}\}$.  
Now $x_2$ and $x_{4L+4}$ both lie outside $X_4^{4L+2} \subset M_L$,  
since otherwise $M_L \times N$ contains a chain longer than $4L$.  Moreover $x_2 \ne x_{4L+4}$,
since otherwise we can form a cycle (of length $4L+2$), hence an infinite chain in $M_{L+1} \times N$, 
contradicting the length bound already established.  Similarly,  $y_1 \ne y_{4L+5}$ are disjoint from $Y_3^{4L+3}$,
since otherwise we can extract a cycle and build an infinite chain in $M_{L+1} \times N$.
%
\end{proof}


 

\subsection{Costs on manifolds:}

\begin{lemma}[Diffeomorphism from interior of simplex to punctured sphere]
\label{L:simplex onto sphere}
Fix the standard simplex $\Delta = \{(t_0,\dots,t_m) \mid t_i\ge 0 \mbox{\rm\ and}\ \sum_{i=0}^m t_i =1 )\}$ and unit ball $\Omega = B_1(e_1) \subset \R^{m+1}$ centered at $e_1 = (1,0,\ldots,0) \in \R^{m+1}$.
There is a smooth map $E:\Delta \longrightarrow \partial \Omega$ which acts as a diffeomorphism
from $\Delta \setminus \partial \Delta$ to $\partial \Omega \setminus \{0\}$
such that $E$ and all of its derivatives vanish on the boundary $\partial \Delta$
of the simplex: $E(\partial \Delta) = \{0\}$.
\end{lemma}

\begin{proof}
Let $f:[0,1] \rightarrow [1,2]$ be a smooth function satisfying the following properties:
\begin{itemize}
\item[(a)] $f$ is  nondecreasing, 
\item[(b)] $f(s)=1$ for every $s\in [0,1/2]$,
\item[(c)] $f(1)=2$ and all the derivatives of $f$ at $s=1$ vanish.
\end{itemize} 
Denote by $D_m$ the closed unit disc of dimension $m$ and by $\bfS^m \subset \R^{m+1}$ the unit sphere. We also denote by $\exp_N : T_N\bfS^{m} \rightarrow \bfS^m$ the exponential mapping from the north pole $N=(0, \ldots, 0,1)$ associated with the restriction of the Euclidean metric in $\R^{m+1}$ to $\bfS^m$. Then we set
$$
F(v) = \exp_N \left[  \frac{\pi}{2} \, f \left(|v|\right)v\right] \qquad \forall v \in D_m.
$$
By construction, $F$ is smooth on $D_m$, $F$ is a diffeomorphism from $\mbox{Int} (D_m)$ to $\bfS^m \setminus \{S\}$, where $S$ denotes the south pole of $\bfS^{m}$, $F(\partial D_m) = \{S\}$ and all the derivatives of $F$ on $\partial D_m$ vanish. Therefore, in order to prove the lemma, it is sufficient to construct a Lipschitz mapping $G: \Delta \rightarrow D_m$ which is smooth on $\mbox{Int}(\Delta)$, is a diffeomorphism from $\mbox{Int}(\Delta)$ to $\mbox{Int}(D_m)$, and sends $\partial \Delta$ to $\partial D_m$.

The  simplex $\Delta$ is contained in the affine hyperplane 
$$
H =  \left\{ \lambda = \left( t_0, \ldots, t_m\right) \, \vert \, \sum_{i=0}^m t_i=1\right\}.
$$
Let $\bar{t}:=(1/(m+1), \ldots, 1/(m+1))$ be the center of $\Delta$, we check easily that $\Delta$ is contained in the disc centered at $\bar{t}$ with radius $\sqrt{1-1/(m+1)}$. For every $t\in \Delta\setminus \{\bar{t}\}$, we set $u_{t} := (t - \bar{t})/|t -\bar{t}|$ and 
$$
\rho(t) := \min \Bigl\{s\geq 0 \, \vert \, \lambda + s \, u_{t}   \in \partial \Delta \Bigr\} \qquad \forall t \in \Delta \setminus \{\bar{t} \}.
$$
By construction, the function $\rho : \Delta \setminus \{\bar{t} \} \rightarrow [0,+\infty)$ is locally Lipschitz and satisfies for every unit vector $u\in \bfS^m \cap H_0$ (with $H = \bar{t} + H_0$),
$$
\rho \left( \bar{t} + \alpha u \right) =  \alpha_u - \alpha  \qquad \forall \alpha \in (0,\alpha_u],
$$
where $\alpha_u>0$ is the unique $\alpha>0$ such that  $\bar{t} + \alpha u \in \partial \Delta$. We note that since $\Delta \subset \bar{B} (\bar{t},\sqrt{1-1/(m+1)})$, we have indeed $\alpha_u \in (0,\sqrt{1-1/(m+1)}]$ for every $u\in  \bfS^m \cap H_0$.  We also observe that the $m$-dimensional ball $H \cap \bar{B}\left(\bar{t},1/(2(m+1))\right)$ is contained in the interior of $\Delta$ and that $\rho\geq 1/(2(m+1))$ on that set. Pick a smooth function $g:[0,+\infty) \rightarrow [0,1]$  satisfying the following properties:
\begin{itemize}
\item[(d)] $g$ is nonincreasing, 
\item[(e)] $g(s)=1-3(m+1)s$ for every $s\in [0,1/(4(m+1))]$,
\item[(f)] $g(s)=0$ for every $s\geq 1/(2(m+1))$.
\end{itemize} 
Let $D$ be the $m$-dimensional unit disc in $H$ centered at $\bar{t}$,  define the function $G^0:\Delta \rightarrow D$ by
$$
G^0(t) = \bar{t} + \left[ 1-  g\left( \rho(t)\right)\right] \, \left( t - \bar{t}\right) +   g\left( \rho(t)\right) \, u_{t}.  
$$
By construction, $G^0$ is Lipschitz and smooth on each ray starting from $\bar{t}$.  Namely, for each unit vector $u\in \bfS^m \cap H_0$, we have 
$$
G_u^0 (\alpha) :=G^0\left( \bar{t} + \alpha u \right) = \bar{t} + \left[ 1-  g\left(    \alpha_u -\alpha    \right)\right] \, \left( \alpha \, u \right) +   g\left( \alpha_u -\alpha\right) \, u \qquad \forall \alpha \in (0,\alpha_u].
$$
The derivative of $G^0$ on each ray $\bar{t} + \R^+ \, u$ is given by
$$
\frac{\partial G^0_u}{\partial \alpha} (\alpha) = \left[  1-  g\left(     \alpha - \alpha_u     \right) +\left(  \alpha -1\right)  g' \left(     \alpha - \alpha_u     \right) \right] \, u \qquad \forall \alpha \in (0,\alpha_u],
$$
and there holds
\begin{eqnarray*}
 & \quad & 1-  g\left(     \alpha - \alpha_u     \right) +\left(  \alpha -1\right)  g' \left(     \alpha - \alpha_u     \right) \\
 & \geq & 1-  g\left(     \alpha - \alpha_u     \right) +\left( \sqrt{1-\frac{1}{m+1}} -1\right)  g' \left(     \alpha - \alpha_u     \right) \\
 & \geq & 1-  g\left(     \alpha - \alpha_u     \right) +\left( \frac{-1}{2(m+1)}\right)  g' \left(     \alpha - \alpha_u     \right) \geq \frac{3}{4},
\end{eqnarray*}
by (d)-(f). Moreover, for every $u\in \bfS^m \cap H_0$, the ray $\bar{t} + \R^+ \, u$ in invariant by $G^0$, $G_u^0(\alpha_u) = \bar{t}+u$, and in addition $G^0(t)=t$ whenever $t \in H\cap \bar{B}\left(\bar{\lambda},1/(2(m+1))\right)$. In conclusion, $G^0$ is Lipschitz and bijective from $\Delta$ to $D$. If we work in polar coordinates $z=(\alpha,u)$ with $\alpha>0$ and $u\in \bfS^{m} \cap H_0$, then $G^0$ reads 
$$
\tilde{G}^0(z) = \tilde{G}^0(\alpha,u) = \left( G_{u} (\alpha), u \right),
$$
for every $z \in \tilde{\Delta}$, the domain of $G^0$ in polar coordinates (since $G^0$ coincides with the identity near $\bar{\lambda}$ we do not care about the singularity at $\alpha=0$). Thus for every $z$  in the interior of $\tilde{\Delta}$ where $\tilde{G}^0$ is invertible, the Jacobian matrix of $\tilde{G}^0$ at $z$, $J_z\tilde{G}^0$ is triangular and invertible. Recall that for every $z$ in the interior of  $\tilde{\Delta}$, the generalized Jacobian of $\tilde{G}^0$ at $z$ is defined as
$$
\mathcal{J}_z \tilde{G}^0 := \mbox{conv} \Bigl\{  \lim_{k} J_{z_k} \tilde{G}^0 \, \vert \, z_k \rightarrow_{k} z, \, G \mbox{ diff. at } z_k            \Bigr\}.
$$
By the above discussion and Rademacher's Theorem, for every $z$ in the interior of  $\tilde{\Delta}$,  $\mathcal{J}_z \tilde{G}^0$ is always a nonempty compact subset of $M_m(\R)$ which contains only invertible matrices. In conclusion, for every $t \in \mbox{Int} (\Delta)$ the generalized Jacobian of $G^0$ at $t$ satisfies the same properties, it is a nonempty compact subset of $M_m(\R)$ which contains only invertible matrices. Thanks to the Clarke Lipschitz Inverse Function Theorem \cite{clarke83}, we infer that the Lipschitz mapping $G^0:\Delta \rightarrow D$ is  locally bi-Lipschitz from $\mbox{Int} (\Delta)$ to $\mbox{Int} (D)$. It remains to smooth $G^0$ in the interior of $\Delta$ by fixing $G^0$ on the boundary $\partial \Delta$. 

To this aim, consider a mollifier $\theta : \R^m \rightarrow \R$, that is a smooth function satisfying the following three conditions: 
\begin{itemize}
\item[(g)] $\theta \geq 0$, 
\item[(h)] $\mbox{Supp} (\theta) \subset \bar{D}_m$,
\item[(i)] $\int_{\R^m} \theta (x) \, dx =1$.
\end{itemize} 
The multivalued mapping $\lambda \in \mbox{Int} (\Delta) \mapsto \mathcal{J}_{\lambda}G^0 \in M_m(\R)$ is uppersemicontinuous (its graph is closed in $\mbox{Int}(\Delta) \times M_m(\R)$) and is valued in the set of compact convex sets of invertible matrices. Hence, there is a continuous function $\epsilon :\mbox{Int} (\Delta) \rightarrow (0,\infty)$ such that for every $t \in \mbox{Int} (\Delta)$ and every matrix $A \in M_m(\R)$, the following holds 
\begin{eqnarray}\label{dim1}
d\left( A, \mbox{conv} \left( \left\{ \mathcal{J}_{\beta}G \, \vert \, \beta \in B(\lambda,\epsilon (t) \right\} \right)  \right) < \epsilon (t) \quad \Longrightarrow \quad A \mbox{ is invertible}.
\end{eqnarray}
Consider also a smooth function $\nu :H \rightarrow \R^+$ such that:
\begin{itemize}
\item[(j)] $\nu(t)=0,$ for every $t \notin \mbox{Int} (\Delta)$, 
\item[(k)] $0< \nu(t) < \min \left\{ d(t, \partial \Delta), \epsilon(t) \right\}$, for every $t \in \mbox{Int} (\Delta)$,
\item[(l)] for every $t \in \mbox{Int} (\Delta)$, $\left| \nabla_t \nu\right| \leq \epsilon(t)/K$, where $K>0$ is a Lipschitz constant for $G^0$.
\end{itemize} 
Then, we define the function $G:\Delta \rightarrow D$ by (we identify $H_0$ with $\R^m$)
$$
G(t) = \int_{\R^m} \theta (x ) \, G^0 \left( t-\nu(t)x\right) \, dx \qquad \forall t \in \Delta.
$$
By construction, $G$ is Lipschitz on $\Delta$,  it coincides with $G^0$ on $ \partial \Delta$, it satisfies $G(\mbox{Int} (\Delta)) \subset \mbox{Int}(D)$, and it is smooth on $\mbox{Int} (\Delta)$. For every $t \in \mbox{Int} (\Delta)$, its Jacobian matrix at $t$ is given by
$$
J_{t} G = \int_{\R^m} \theta (x) \, J_{t-\nu(t)x} G^0 \cdot \left( I_n -x \cdot \bigl(\nabla_{t} \nu\bigr)^* \right) \, dx.
$$
Hence, we have for every $t \in \mbox{Int} (\Delta)$,
$$
\left\| J_{t} G - \int_{\R^m} \theta (x) \, J_{t-\nu(t)x} G  \, dx\right\| \leq K \, \left|  \nabla_t \nu \right|
$$
and
$$
 \int_{\R^m} \theta (x) \, J_{t-\nu(t)x} G  \, dx \in \mbox{conv} \Bigl\{ J_{\beta} G \, \vert \, \beta \in B(t,\nu(t)) \Bigr\}.
$$
Using (\ref{dim1}) and (j)-(l), we infer that $G$ is a local diffeomorphism at every point of $\mbox{Int}(\Delta)$. Moreover, $G$ is surjective.  If not, there is $y\in D$ such that $y$ does not belong to the image of $G$. Since $G=G^0$ on $\partial \Delta$, $y$ does not belong to $\partial D$. Thus there is $y' \in \partial G(\Delta) \setminus \partial D$. Since $G$ is a local diffeomorphism at any preimage of $y'$, we get a contradiction. In conclusion, $G$ is a Lipschitz mapping from $\Delta$ to $D$ which sends bijectively $\partial \Delta$ to $\partial D$, which sends $\mbox{Int} (\Delta)$ to $\mbox{Int}(D)$, which is  surjective, and which is a smooth local diffeomorphism at every point of $\mbox{Int} (\Delta)$. Moreover, $D$ is simply connected. Hence $G:\Delta \rightarrow D$ is a Lipschitz mapping which is a smooth diffeomorphism from $\mbox{Int} (\Delta)$ to $\mbox{Int}(D)$. We conclude easily.
\end{proof}
\begin{proposition}[Smooth costs on arbitrary manifolds leading to unique optimal transport]
Fix smooth closed manifolds $M, N$. 
Then there exists a cost 
$c\in C^\infty(M \times N)$ such that: for any pair of Borel probability measures $\mu$ on $M$
and $\nu$ on $N$ which charge no $c-c$ hypersurfaces in their respective domains,
the minimizer of \eqref{Kantorovitch inf} is unique.
\end{proposition}

\begin{figure}[!h]
\label{FIGcostonmanifold}
\centering
\begin{pdfpic}
\centering 
\psset{unit=0.75cm}
\begin{pspicture}(12,11)
\pscircle(6,4){2}
\pscircle(6,3){1}
\pscircle(6,5){3}
\pscircle(6,6){4}
\psdots[dotsize=3pt](6,2)
\rput[c](6.65,1.65){$O$}
\psline[linewidth=0.25pt,linecolor=lightgray](0,1)(12,1)(12,11)(0,11)(0,1) 
\end{pspicture}
\end{pdfpic}
\caption{A bouquet of nested convex sets}
\end{figure}
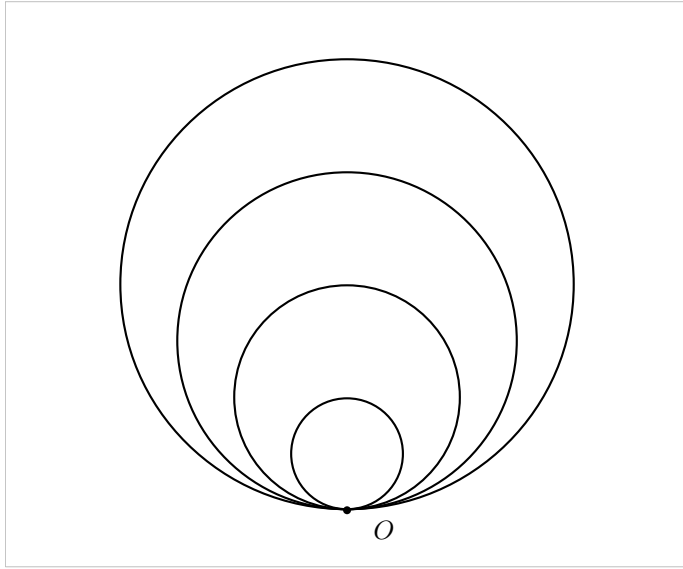


\begin{proof}
Let $m$ and $n$ denote the dimensions of $M$ and $N$,  and assume $m\ge n$ without loss
of generality.  Due to their smoothness,  it is a classical result that both manifolds admit 
smooth triangulations  
\cite{Whitney57}
into finitely many (say $k_M$ and $k_N$) simplices
(by compactness).

For each $k \in \{1,2,\ldots,k_M\}$,  dilating the map $E$ of Lemma \ref{L:simplex onto sphere}
by a factor of $k$  induces a smooth map from the $k$-th simplex of $M$ to the sphere 
$k \partial B_1(e_1)$ of radius $k$ centered at $(k,0,\ldots,0) \in \R^{m+1}$.  Taken together,
these $k_M$ maps define a single smooth map 
$E_M: M \longrightarrow \tilde M$ where 
$\tilde M = \bigcup_{k=1}^{k_M} \partial \Omega_k$
and $\Omega_k = k B_1(e_1) \subset \R^{m+1}.$
This map acts as a diffeomorphism from the union of simplex interiors in $M$ 
to $\tilde M \setminus \{\0_{m+1}\}$ while collapsing their boundaries onto the origin  $\0_{m+1}$
in $\R^{m+1}$.
Set $M_0 := E_M^{-1}(\0_{m+1})$.

Define the analogous map $E_N: N \longrightarrow \tilde N \subset \R^{n+1}$ where
$\tilde N = \bigcup_{k=1}^{k_N} k \partial B_1(e_1) \subset \R^{n+1}$ and $N_0 = E_N^{-1}(\0_{n+1})$.
In case $n<m$,  we embed $\R^{n+1}$ into $\R^{m+1}$
by identifying $\R^{n+1}$ with $\{(x_1,\ldots,x_{m+1}) \in \R^{m+1} \mid 
x_{n+2}=\cdots=x_{m+1}=0\}$.  

The cost 
$$
c(x,y) := |E_M(x) - E_N(y)|^2/2
$$
on $M \times N$ then satisfies the conclusions of the proposition.  Its smoothness follows 
from that of $E_M$ and $E_N$.  Lemma \ref{lemmanested} shows that no 
chains of length greater than $4k_M$ lie in $(M\setminus M_0) \times 
(N\setminus N_0)$.  On the other hand,  the simplex boundaries $M_0$ lies in a finite union of smooth hypersurfaces, hence are $\mu$-negligible.  Similarly,  $N_0$ is $\nu$-negligible. The desired conclusion 
now follows from Theorem~\ref{THM1}.
\end{proof}

\section{Preliminaries on numbered limb systems}\label{SEClimb}

\subsection{Classical numbered limb systems}\label{cnls}

The concept of numbered limb system was introduced by Hestir and Williams in \cite{hw95}.
Like Benes and Stepan \cite{BenesStepan87},  their aim was to find necessary and sufiicient conditions
on the support of a joint measure to guarantee its extremality in the space of measures which share its
marginals.

\begin{definition}[Numbered limb system]
\label{DEFlimb}
Let $X$ and $Y$ be subsets of complete separable metric spaces. A relation $S \subset X\times Y$ is a numbered limb system if there are countable disjoint decompositions of $X$and $Y,$ 
$$
X=\bigcup_{i=0}^{\infty} I_{2i+1} \quad \mbox{and} \quad Y=\bigcup_{i=0}^{\infty} I_{2i},
$$
with a sequence of mappings 
$$
f_{2i} : \mbox{Dom}(f_{2i}) \subset Y \longrightarrow X \quad \mbox{and}  \quad f_{2i+1} : \mbox{Dom}(f_{2i+1}) \subset X \longrightarrow Y
$$
such that 
\begin{eqnarray}\label{DEFlimb1}
S= \bigcup_{i=1}^{\infty} \Bigl( \mbox{Graph} (f_{2i-1}) \cup \mbox{Antigraph}(f_{2i}) \Bigr)
\end{eqnarray}
and
\begin{eqnarray}\label{DEFlimb2}
\mbox{Dom} (f_k) \cup \mbox{Ran} (f_{k+1}) \subset I_k \qquad \forall k\geq 0.
\end{eqnarray}
\end{definition}

\begin{figure}[!h]
\label{FIGlimb}
\centering
\begin{pdfpic}
\centering
\psset{unit=1cm}
\begin{pspicture}(-2,-1)(7,8)
\psgrid[subgriddiv=5,subgriddots=10,gridlabels=0](0,0)(5,6)
\rput(0.5,-0.2){\small $I_1$}
\rput(1.5,-0.2){\small $I_3$}
\rput(2.5,-0.2){\small $I_5$}
\rput(3.5,-0.2){\small $I_7$}
\rput(4.5,-0.2){\small $I_9$}
\rput(-0.3,.5){\small $I_{0}$}
\rput(-0.3,1.5){\small $I_{2}$}
\rput(-0.3,2.5){\small $I_{4}$}
\rput(-0.3,3.5){\small $I_{6}$}
\rput(-0.3,4.5){\small $I_{8}$}
\rput(-0.3,5.5){\small $I_{10}$}
\pscurve(4.9,5.9)(4.1,5.5)(4.9,5.1)
\pscurve(4.1,4.9)(4.35,4.5)(4.9,4.9)
\pscurve(3.9,4.9)(3.1,4.6)(3.9,4.3)
\pscurve(2.9,3.9)(2.3,3.75)(2.9,3.5)(2.1,3.3)
\pscurve(1.9,2.9)(1.2,2.6)(1.9,2.4)
\pscurve(0.1,1.9)(0.9,1.5)(0.1,1.1)
\pscurve(3.1,3.9)(3.4,3.6)(3.6,3.6)(3.9,3.9)
\pscurve(2.1,2.9)(2.5,2.2)(2.9,2.9)
\pscurve(1.1,1.3)(1.5,1.4)(1.9,1.9)
\pscurve(0.1,0.9)(0.5,0.4)(0.9,0.9)
\psline[linewidth=0.25pt,linecolor=lightgray](-2,-1)(6,-1)(6,8)(-2,8)(-2,-1)
\end{pspicture}
\end{pdfpic}
\caption{A numbered limb system with $N=10$}
\end{figure}
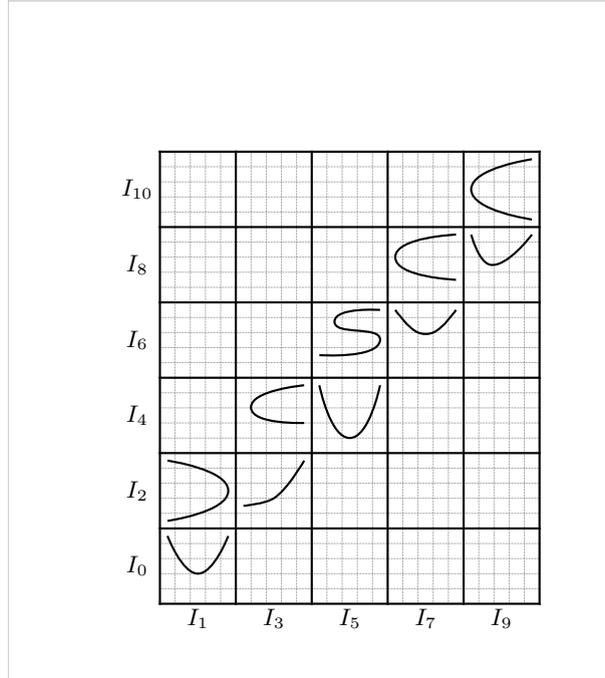

The following statement from \cite{akm11}
extends and relaxes a result of Hestir and Williams \cite{hw95}.  
Here $\pi^X(x,y)=x$ and $\pi^Y(x,y) =y$.

\begin{theorem}[Measures on measurable numbered limb systems are simplicial]
\label{THMlimb}
Let $X$ and $Y$ be Borel subsets of complete separable metric spaces, equipped with $\sigma$-finite Borel measures $\mu$ on $X$ and $\nu$ on $Y$. Suppose there is a numbered limb system 
\begin{equation}\label{NLS}
S= \bigcup_{i=1}^{\infty} \Bigl( \mbox{Graph} (f_{2i-1}) \cup \mbox{Antigraph}(f_{2i})\Bigr) 
\end{equation}
with the property that $\mbox{Graph} (f_{2i-1})$ and $ \mbox{Antigraph}(f_{2i}) $ are $\gamma$-measurable subsets of $X\times Y$ for each $i\geq 1$ and for every $\gamma\in \Gamma (\mu,\nu)$ vanishing outside of $S$. If the system has finitely many limbs or $\mu[X]<\infty$, then at most one $\gamma \in \Gamma (\mu,\nu)$ vanishes outside of $S$. If such a measure exists, it is given by $\gamma =\sum_{k=1}^{\infty} \gamma_k$ where for every $i\geq 1$,
\begin{eqnarray}\label{descent}
\left\{
\begin{array}{l}
\gamma_{2i-1}  =  \bigl( \mbox{id}_X \times f_{2i-1}\bigr)_{\sharp} \eta_{2i-1},  \qquad \gamma_{2i}  =   \bigl(f_{2i} \times \mbox{id}_Y \bigr)_{\sharp} \eta_{2i}, \\
\eta_{2i-1}  =  \left( \mu- \pi^X_{\sharp} \gamma_{2i} \right)|_{\Dom f_{2i-1}}, \quad \eta_{2i}  = \left( \nu - \pi^Y_{\sharp} \gamma_{2i+1} \right)|_{\Dom f_{2i}}.
\end{array}
\right.
\end{eqnarray}
Here $\eta_k$ is a Borel measure on $I_k$ and $f_k$ is measurable with respect to the $\eta_k$ completion of the Borel $\sigma$-algebra. If the system has $N<\infty$ limbs, $\gamma_k=0$ for $k>N$, and $\eta_k$ and $\gamma_k$ can be computed recursively from the formula above starting from $k=N$.
\end{theorem}




The statement of Theorem \ref{THMlimb} from Ahmad, Kim and McCann, 
like its antecedent in \cite{hw95}, give a sufficient condition for extrememality.   It is separated from Benes and Stepan \cite{BenesStepan87} and Hestir and Williams' \cite{hw95}  necessary conditions for extremality by the $\gamma$-measurability assumed
for the graphs and antigraphs (which is satisfied, for example,  whenever the graphs and antigraphs are Borel.)
For sets $S$ of the form \eqref{NLS} whose graphs and antigraphs fail to be measurable,  there may exist non-extremal
measures vanishing outside of $S$, as shown by Hestir and Williams using the axiom of choice \cite{hw95}.
Such issues are further explored by Bianchini and Caravenna \cite{bc09} and Moameni \cite{Moameni14p},  who arrive at their own criteria for extremality.
Moameni's is closest in spirit to the approach developed below based on chain length: he gets his measurability by assuming the existence of a measurable Lyapunov function to distinguish different levels of the dynamics.

\section{Proof of Theorem \ref{THM1} and Remark \ref{R singular measures}}\label{SECproofsTHM1}


Since the source and target spaces are closed manifolds and the cost $c \in C^1$, 
Gangbo and McCann~\cite{gm96} provide
a $c$-cyclically monotone compact set $\mathcal{S} \subset M\times N$ and Lipschitz potentials $\psi:M\rightarrow \R$ and $\phi:M \rightarrow \R$ which satisfy 
\begin{eqnarray}\label{pot1}
\psi(x) = \max \Bigl\{ \phi(y)-c(x,y) \, \vert \, y \in N\Bigr\} \qquad \forall x \in M,
\end{eqnarray}
\begin{eqnarray}\label{pot2}
\phi(y) = \min \Bigl\{ \psi(x)+c(x,y) \, \vert \, x \in M\Bigr\} \qquad \forall y \in N,
\end{eqnarray}
\begin{eqnarray}\label{potincl}
\mathcal{S} \subset \partial_c\psi  :=  \Bigl\{ (x,y) \in M\times N \, \vert \, c(x,y)= \phi(y)-\psi(x)\Bigr\},
\end{eqnarray}
such that any plan $\gamma \in \Pi(\mu,\nu)$ is optimal if and only if $\mbox{Supp}(\gamma)\subset \mathcal{S}$.   Indeed,  we henceforth $\mathcal S$ to be the smallest compact set with these properties.

We recall that the $c$-subdifferential of $\psi$ at $x\in M$ and the $c$-superdifferential of $\phi$ at $y\in N$ are defined using \eqref{potincl}:
\begin{eqnarray*}
\partial_c\psi (x) &:=& \Bigl\{ y \in N \, \vert \, (x,y) \in \partial_c \psi \Bigr\} \\
\mbox{and} \quad \partial^c\phi (y) &:=& \Bigl\{ x \in M \, \vert \, (x,y) \in \partial_c\psi \Bigr\}.
\end{eqnarray*}
Note that since both $\psi$ and $\phi$ are Lipschitz and $\mu$ and $\nu$ are both absolutely continuous with respect to Lebesgue, thanks to Rademacher's theorem, $\psi$ and $\phi$ are differentiable almost everywhere with respect to $\mu$ and $\nu$ respectively. Let $\Dom d\psi$ denotes the subset of $M$ on which $\psi$ is differentiable.
Following Clarke \cite{Clarke75}, for every $x \in M$ (resp. $y\in N$), we denote by $D^* \psi (x)$ and $\partial \psi(x)$ (resp.  $D^* \phi (y)$ and $\partial \phi(y)$) the limiting and  generalized differentials of $\psi$ at $x$ (resp. $\phi$ at $y$) which are defined by (we proceed in the same way with $\phi$)
$$
D^*\psi(x) = \Bigl\{ \lim_{k\rightarrow \infty} p_k \, \vert \, p_k = d\psi(x_k), x_k \rightarrow x, x_k \in \Dom d\psi\Bigr\} \subset T_x^*M,
$$ 
and 
$$
\partial \psi (x) = \mbox{conv} \left( D^*\psi(x)\right) \subset T_x^*M.
$$
By Lipschitzness, for every $x\in M$, the sets $D^*\psi(x)$ and $\partial \psi$ are nonempty and compact, and of course $\partial \psi(x)$ is convex.   The next three propositions are relatively standard;
the lemmas which follow them are new.

\begin{proposition}\label{P Clarke}
For $c \in C^1$, the potentials $\psi$ and $\phi$ of \eqref{pot1}-\eqref{pot2} satisfy:
\begin{itemize}
\item[(i)] The mappings $x\in M \mapsto \partial \psi(x)$ and $y\in N \mapsto \partial \phi(y)$ have closed graph.
\item[(ii)] For every $x\in M$, $\psi$ is differentiable at $x$ if and only if $\partial \psi(x)$ is a singleton.
\item[(iii)] For every $y\in N$, $\phi$ is differentiable at $y$ if and only if $\partial \phi(y)$ is a singleton.
\item[(iv)] The singular sets $M_0:=M \setminus \dom d\psi$ and $N_0 := \setminus \dom d\phi$ are $\sigma$-compact.
\end{itemize}
\end{proposition}

\begin{proof}[Proof of Proposition \ref{P Clarke}]
Assertion (i) is well-known \cite{Clarke75}, and follows easily from the definitions of $\partial \psi$ and $\partial \phi$. Let $x\in M$ be such that $\psi$ is differentiable at $x$. From \eqref{pot1}--\eqref{potincl} we have
\begin{eqnarray}\label{Clarke1}
-\frac{\partial c}{\partial x}(x,y) = d\psi(x) \qquad \forall y \in \partial_c \psi(x).
\end{eqnarray}
Argue by contradiction and assume that $\partial \psi(x)$ is not a singleton. This means that $D^*\psi(x)$ is not a singleton too, let $p\neq q $ be two one-forms in $D^*\psi(x)$. Then there are two sequences $\{x_k\}_k, \{x'_k\}_k$ converging to $x$ such that $\psi$ is differentiable at $x_k$ and $x_k'$ and 
$$
\lim_{k\rightarrow \infty} d\psi(x_k)=p, \qquad \lim_{k\rightarrow \infty} d\psi({x_k'})=q.
$$
For each $k$, there are $y_k \in \partial_c \psi(x_k)$, $y_k' \in \partial_c \psi (x_k')$ such that
$$
-\frac{\partial c}{\partial x}\bigl(x_k,y_k\bigr) = d\psi(x_k) \qquad \mbox{and} \qquad -\frac{\partial c}{\partial x}\bigl(x_k,y_k'\bigr) = d\psi({x_k'}).
$$
By compactness of $N$, we may assume that the sequences $\{y_k\}_k, \{y_k'\}_k$ converge respectively to some $\bar{y} \in \partial_c \psi(x)$ and $\bar{y}'\in \partial_c \psi(x)$. Passing to the limit , we get 
$$
-\frac{\partial c}{\partial x}\bigl(x,\bar{y}\bigr) = p \qquad \mbox{and} \qquad -\frac{\partial c}{\partial x}\bigl(x,\bar{y}'\bigr) = q,
$$
which contradicts (\ref{Clarke1}). On the other hand, if $\partial \psi(x)$ is a singleton, then $\psi$ is differentiable at $x$ (indeed, $d\psi: \Dom d\psi \longrightarrow T^*M$ is continuous at $x$,
so $x$ is a Lebesgue point for $d\psi \in L^\infty_{loc}$). 

(iv) The set of $x$ 
such that $\partial \psi(x)$ 
is a singleton is $\sigma$-compact because the multi-valued mapping 
$x\mapsto \partial \psi(x)$ 
had s closed graph, and the mapping 
$x \mapsto \mbox{diam} (\partial \psi(x))$ 
is upper semicontinuous. For every 
whole number $q$, this implies those $x$ with 
$\mbox{diam} (\partial \psi(x)) \ge 1/q$ 
form closed subset of the compact manifold $M$. 
The singular set $M\setminus \Dom d\psi$ is the union of such subsets over $q=1,2,\ldots$.
$\sigma$-compactness of $N\setminus \Dom d\phi = 
\bigcup_{q=1}^\infty \{y \in N \mid {\rm diam} (\partial\phi(y)) \ge 1/q\}$ follows by symmetry.
\end{proof}


\begin{proposition}[Differentiability a.e.]
\label{PROP1}
The sets $M_0 := M \setminus \dom d\psi$, $N_0 := N\setminus \dom d\phi$
and $M_0 \times N_0$ are $\sigma$-compact, and
$\mu(M_0)=\nu(N_0)=\gamma(M_0\times N_0)=0$  for every plan $\gamma \in \Pi(\mu,\nu)$.
\end{proposition}

\begin{proof}[Proof of Proposition \ref{PROP1}]
Since $\mathcal{S}$ is compact, the $\sigma$-compactness of $E_0$ follows from that 
shown in Proposition \ref{P Clarke}(iv) for $M_0$ and $N_0$ (a product of unions being the union of the products). 
If $\mu$ and $\nu$ are absolutely continuous with respect to Lesbesgue,
Rademacher's theorem asserts $\mu[M_0]=0$ and $\nu[N_0]=0$.
Otherwise $c \in C^{1,1}$, in which case Gangbo and McCann show the potentials $\phi$ and $-\psi$ are 
semiconvex \cite{gm96},  meaning
 their distributional Hessians admit local bounds from below in $L^\infty$.  
In this case the conclusion of Rademacher's theorem can be sharpened: 
Zajicek \cite{Zajicek79} shows $M_0$ and $N_0$
to be contained in countably many $c-c$ hypersurfaces,  on which $\mu$ and $\nu$ are assumed to vanish.  Finally
$\gamma(M_0 \times N_0) \le \gamma(M_0 \times N) = \mu(M_0)=0$.
\end{proof}


Since our manifolds $M$ and $N$ are compact,  any open subset is 
$\sigma$-compact; in particular the complement of $\mathcal S$ is $\sigma$-compact.
In view of this fact and the proposition preceding it, by enlarging $E_0$ if necessary 
we may henceforth assume (i) $(M \times N \setminus {\mathcal S}) \subset E_0$ and (ii)
$M_0 \times N_0 \subset E_0$. Then $\tilde{\mathcal{S}}:= M \times N \setminus E_0$
ensures that 
for all pairs $(x,y) \in \tilde{\mathcal{S}} := M \times N \setminus E_0$  we have differentiability of $\psi$ at $x$ and of $\phi$ at $y$.

\begin{proposition}[Marginal cost is marginal price]
\label{PROP2}
For every $(x,y)\in \tilde{\mathcal{S}}$, \eqref{pot1}--\eqref{potincl} imply
\begin{eqnarray}\label{diff}
d\psi(x) = - \frac{\partial c}{\partial x} (x,y) \quad \mbox{and} \quad d\phi(x) = \frac{\partial c}{\partial y} (x,y).
\end{eqnarray} 
\end{proposition} 

\begin{proof}[Proof of Proposition \ref{PROP2}]
Let $(x,y) \in \tilde{\mathcal{S}}$, then we have by (\ref{pot1})--(\ref{pot2}),
\begin{eqnarray*}
 \phi(y') -c(x,y')\leq \psi(x) \quad \forall y' \in N \quad & \mbox{and} & \quad  \phi(y) -c(x,y) = \psi(x) \\
 \psi(x')+c(x',y) \geq \phi(y) \quad \forall x' \in M  \quad & \mbox{and} & \quad  \psi(x) +c(x,y) = \phi(y).
 \end{eqnarray*}
We conclude easily since both $\psi$ and $\phi$ are differentiable respectively at $x$ and $y$.
\end{proof}

We call $L$-chain in $\tilde{\mathcal{S}}$ any ordered family of pairs 
$$
 \Bigl( (x_1,y_1), \ldots (x_L,y_L)\Bigr) \subset \tilde{\mathcal{S}}^L
$$
such that for every $l=1, \ldots, L-1$ there holds, either 
$$
\left\{
\begin{array}{rclcl}
x_{l} & = & x_{l+1}, &&\\
 y_l & \neq & y_{l+1} &=& y_{\min\{L,l+2\}}
 \end{array}
 \right. 
  \quad  \mbox{ or } \quad
 \left\{
\begin{array}{rclcl}
 y_l & = & y_{l+1}, \\
 x_l & \neq & x_{l+1} &=& x_{\min\{L,l+2\}} .
  \end{array}
 \right. 
$$
Note that by construction, the set of pairs of any  $L$-chain in $\tilde{\mathcal{S}}$ is $c$-cyclically monotone as a subset of $\mathcal{S}$,  so by (\ref{diff}), any $L$-chain in $\tilde{\mathcal{S}}$ is indeed an $L$-chain with respect to $c$ (Definition \ref{Laltchain}). 
 We define the level $\ell(x,y)$ of each $(x,y) \in \tilde{\mathcal{S}}$ 
to be the supremum of all natural numbers $L\in\N^*$ such that there is at least one chain $\left( (x_1,y_1), \ldots (x_L,y_L) \right)$ in  $\tilde{\mathcal{S}}$ of length $L$  such that $(x,y) = (x_L,y_L)$. Moreover, given a chain $\left( (x_1,y_1), \ldots (x_L,y_L) \right)$ with $L\geq 2$ in $\tilde{\mathcal{S}}$, we say that $(x_L,y_L)$  is a horizontal end if $y_L=y_{L-1}$ and a vertical end if $x_
L=x_{L-1}$. We set
$$
\tilde {\mathcal{S}}_L := \Bigl\{ (x,y) \in \tilde {\mathcal{S}} \, \vert \, \ell (x,y) \ge L \Bigr\} \qquad \forall L \in \N^*,
$$
and denote  by $\tilde{\mathcal{S}}_L^h$ (resp. $\tilde{\mathcal{S}}_L^v$)  the set of pairs  $(x,y) \in \tilde{\mathcal{S}}_L$ such that there exists a $L$-chain $ \left( (x_1,y_1), \ldots (x_L,y_L) \right)$ in  $\tilde{\mathcal{S}}$  such that $(x,y) = (x_L,y_L)$ and $y_{L-1}=y_L$ (resp. $x_{L-1}=x_L$). Although projections of Borel sets are not necessarily Borel (see \cite{srivastava98}), the following lemma holds.


\begin{lemma}[Borel measurability]\label{LEM3}
The sets $\tilde{\mathcal{S}}_1=\tilde{\mathcal S}$ 
and $\tilde{\mathcal{S}}_2^h, \tilde{\mathcal{S}}_2^v, \ldots, \tilde{\mathcal{S}}_L^h, \tilde{\mathcal{S}}_L^v$ are Borel:  each 
takes the form $\cup_{p=1}^\infty \cap_{q=1}^\infty U_{p,q}$,  where the sets $U_{1,1}$ and
$U_{p-1,q} \subset U_{p,q} \subset U_{p,q-1}$ are open for each $p,q\ge 2$.
\end{lemma}

\begin{proof}[Proof of Lemma \ref{LEM3}]
Given $L \ge 3$ odd, 
we shall show that $\tilde{\mathcal{S}}_L^h$ has the asserted structure.
The other cases are left to the reader. Endow the manifolds $M$ and $N$ 
with Riemannian distances $d_M$ and $d_N$, and let $d_L$ denote the product
distance on the product manifold $(M\times N)^L$.
For every integer $p\geq 1$, denote by $S_p$ the set of $L$-tuples
$$
 \Bigl( (x_1,y_1), \ldots (x_L,y_L)\Bigr) \subset \mathcal{S}^L
$$
satisfying for every $i=1, \ldots, L-1$,
$$
\left\{ \begin{array}{lllll}
\mbox{for $l$ even}: & x_{l+1} = x_l & \mbox{and} & d_N(y_{l+1},y_l ) \geq 1/p, \\
\mbox{for $l$ odd:} &y_{l+1} = y_l & \mbox{and} & d_M(x_{l+1},x_l ) \geq 1/p. \\
\end{array}
\right.
$$
Since $\mathcal{S}$ is compact, the set $S_p$ is compact too.  

On the other hand,  $\sigma$-compactness of $E_0$ yields 
$(M \times N) \setminus E_0 = \bigcap_{q=1}^\infty V_q$ for a monotone sequence of open sets 
$V_q \subset V_{q-1}$ .
For every integer $q\geq 1$, we denote by $S_q'$ the open set of $L$-tuples
$$
 \Bigl( (x_1,y_1), \ldots (x_L,y_L)\Bigr) \subset \bigl( V_q \bigr)^L.
$$
A pair $(x,y) \in M\times N$ belongs to $\tilde{\mathcal{S}}_L^h $ if and only if there is $p\geq 1$ such that 
$$
(x,y) \in \mbox{Proj}_L \left( \bigcap_q \left( S_p \cap S_q'\right)\right),
$$
where the projection $ \mbox{Proj}_L : (M\times N)^L \rightarrow \mathcal{S}$ is defined by
$$
 \mbox{Proj}_L \left(  (x_1,y_1), \ldots (x_L,y_L)\right) := (x_L,y_L).
$$
For integers $p, q\geq 1$,  let $S_p^q$ the set of points which are at distance $d_L<1/q$ from $S_p$ in $(M\times N)^L$. Since $S_p$ is compact, for every $p$ we have
$$
 \bigcap_q \left( S_p \cap S_q'\right) =  \bigcap_q \left( S_p^q \cap S_q'\right).  
$$
Moreover since for every $p$, the sequence of sets $\{S_p^q \cap S_q'\}$ is non-increasing with respect to inclusion, we have  for every $p$,
$$
\mbox{Proj}_L \left( \bigcap_q \left( S_p^q \cap S_q'\right)\right) = \bigcap_q \mbox{Proj}_L \left( S_p^q \cap S_q'\right).
$$ 
The open sets $U_{p,q} = \mbox{Proj}_L \left( S_p^q \cap S_q'\right)$ then have the asserted monotonicities $U_{p-1,q} \subset U_{p,q} \subset U_{p,q-1}$ with respect to $p$ and $q$,
and we find $\tilde{\mathcal{S}}_L^h = \bigcup_{p=1}^\infty \bigcap_{q=1}^\infty U_{p,q}$
is the desired countable union of $G_\delta$ sets.
\end{proof}

\begin{corollary}[Borel measurability of projections]\label{CorProj}
For $i \ge 1$, the projections $\pi^M(\tilde{\mathcal{S}}_i)$ and $\pi^N(\tilde {\mathcal S}_i)$ of
 $\tilde{\mathcal{S}}_i$
(and of $\tilde{\mathcal{S}}_i^h, \tilde{\mathcal{S}}_i^v$ if $i>1$) 
take the form $\cup_{p=1}^\infty \cap_{q=1}^\infty V_{pq}$,  where $V_{1,1}$ and
$V_{p-1,q} \subset V_{p,q} \subset V_{p,q-1}$ are open for each $p,q\ge 2$.
\end{corollary}


\begin{proof}
If $\tilde {\mathcal S}^{h/v}_i = \cup_p \cap_q U_{p,q}^{h/v}$ for $i>1$ with $U_{p-1,q}^{h/v} \subset U_{p,q}^{h/v} \subset U_{p,q-1}^{h/v}$ then setting $V_{p,q} = \pi^M(U_{p,q})$ with $U_{p,q} = U^h_{p,q} \cup U^v_{p,q}$ 
shows $\pi^M(\tilde {\mathcal S}_i) =\cup_p \cap_q V_{p,q}$ as desired.  The other cases are similar.
\end{proof}

We recall that a set $S \subset M\times N$ is called a {\em graph} if for every $(x,y)\in S$ there is no $y'\neq y$ such that $(x,y')\in S$. A set $S \subset M\times N$ is called an {\em antigraph} if for every $(x,y)\in S$ there is no $x'\neq x$ such that $(x',y)\in S$. Any graph is the graph of a function defined on a subset of $M$ and valued in $N$ while any antigraph is the graph of a function defined on a subset of $N$ and valued in $M$. We call Borel graph or Borel antigraph any graph or antigraph which is a Borel set in $M\times N$. We are now ready to construct our numbered limb system. \\

Motivated by the inclusion $\tilde {\mathcal S}_{k+1} \subset \tilde {\mathcal S}_k$, we set 
$E_1  
:= \tilde{\mathcal{S}}_1 \setminus \tilde {\mathcal{S}}_{2}$,
\begin{equation}\label{Eks}
E_k := \tilde{\mathcal{S}}_k \setminus \tilde {\mathcal{S}}_{k+1} \mbox{\ and}\ 
\left\{ \begin{array}{l}
E_k^h :=  E_k \cap \tilde{\mathcal{S}}_k^h, \quad E_k^{h-} = E_k \setminus \tilde{\mathcal{S}}_k^v,  \\
E_k^v:= E_k \cap \tilde{\mathcal{S}}_k^v, \quad E_k^{v-} = E_k \setminus \tilde{\mathcal{S}}_k^h, \\
E_k^{hv} := E_k^h \cap E_k^v, 
\end{array}
\right.
\quad \forall k \ge 2.
\end{equation}
Notice that $E_k$ consists precisely of the points in $\tilde{\mathcal S}$ at level $k$.
All these sets are Borel according to Lemma~\ref{LEM3}.
Letting $E_\infty := \bigcap_{k=1}^\infty \tilde{\mathcal S}_k$ gives a decomposition


\begin{eqnarray}\label{20may1}
\tilde{\mathcal{S}} = E_\infty \cup E_1 \cup \left( \bigcup_{k=2}^\infty \bigl( E_k^{h-} \cup E_k^{hv} \cup E_k^{v-}\bigr) \right)
\end{eqnarray} 
of ${\mathcal S}$ into disjoint Borel sets.
The next lemma implies the $E_k^h$ are graphs and the
 $E_k^v$ are antigraphs; $E_1$ is simultaneously a graph and an antigraph,
as are the $E_k^{hv}$.


\begin{lemma}[Graph and antigraph properties]\label{L-key}
(a) Let $(x,y_i) \in E_i$ and $(x,y_j) \in E_j$ with $j \ge i \ge 1$ and $y_i \ne y_j$.  Then $i\ge 2$.
Moreover, if $j>i$ then $(x,y_i) \in E_i^{h}$ and $(x,y_j) \in E_{i+1}^{v-}$ so $j=i+1$;  otherwise $j=i$ and 
both $(x,y_i),(x,y_j) \in E_i^{v-}$.  

(b) Similarly, suppose $(x_i,y) \in E_i$ and $(x_j,y) \in E_j$ with $j \ge i \ge 1$ and $x_i \ne x_j$.
Then $i=2$ and 
if $j>i$ then $(x_i,y) \in E_i^{v}$ and $(x_j,y) \in E_{i+1}^{h-}$so $j=i+1$;  otherwise $j=i$ and 
both $(x,y_i),(x,y_j) \in E_i^{h-}$. 
\end{lemma}

\begin{proof}
(a) Let $(x,y_i) \in E_i$ and $(x,y_j) \in E_j$ with $j \ge i \ge 1$ and $y_i \ne y_j$.
Then 
$(x,y_i), (x,y_j)$ form a $2$-chain and both points lie in $\tilde {\mathcal S}_2$,
forcing $i\ge 2$.   If $(x,y_j) \in E_j^h$, there is a $j$-chain in $\tilde {\mathcal S}$ terminating
in the horizontal end $(x,y_j)$.  Appending $(x,y_i)$ to this chain produces a chain of length $j+1$
with vertical end $(x,y_i)$, whence $i=\ell(x,y_i) \ge j+1$.  This contradicts our hypothesis $i \le j$.  We
therefore conclude
$(x,y_j) \in E_j^{v-}$. 
Note that if $(x,y_i) \in E_i^h$,  the same argument shows
\begin{equation}\label{stepping stone}
j=\ell(x,y_j) \ge i+1.
\end{equation}
Whether or not this is true, ${\mathcal S}$ contains a $j$-chain 
\begin{equation}\label{key lemma chain}
\Bigl( (x_1',y_1'),\ldots, (x_j',y_j')\Bigr)
\end{equation}
terminating in the vertical end $(x,y_j)$, so
 $$
x = x_j'= x_{j-1}', \quad \mbox{and } \quad y_j = y_j' \ne y'_{j-1}.
 $$

Now either (c) $(x,y_i) \in E_i^h$ or (d) $(x,y_i) \in E_i^{v-}$.
In case (c), we claim $y_{j-1}'=y_i$.  Otherwise the sequence
 $$
 \Bigl(  (x_1',y_1'),\ldots, (x_{j-1}',y_{j-1}'), (x,y_i)\Bigr)
 $$
would be a $j$-chain in $\tilde{\mathcal{S}}$ of length
$j \le \ell(x,y_i) =i$, contradicting \eqref{stepping stone}.  
Thus $y_{j-1}'=y_i$ and $i=\ell(x,y_i) \ge j-1$, which implies equality holds in 
\eqref{stepping stone}.

In case (d), $(x,y_i) \in E_i^{v-}$, we replace $(x_j',y_j')$ with $(x,y_i)$
in \eqref{key lemma chain} to produce a chain of length $j \le \ell(x,y_i)=i$,
forcing $i=j$ as desired.

Part (b) of the lemma now follows from part (a) by symmetry.
\end{proof}

We define the graphs and antigraphs of our numbered limb system.
\begin{eqnarray} \label{Gis}
\nonumber G_1 :=& E_1 \cup E_2^{h-}, & \\
G_{2i} :=& E_{2i}^{v-} \cup E_{2i}^{hv} \cup E_{2i+1}^{v-} &= E_{2i}^v \cup E_{2i+1}^{v-}, \qquad {\rm and} \\
\nonumber G_{2i+1} :=& E_{2i+2}^{h-} \cup E_{2i+1}^{hv} \cup E_{2i+1}^{h-} &= E_{2i+2}^h \cup E_{2i+1}^{h-}  
\end{eqnarray}
for all integers $i \in \N^*$, 
and adopt the convention $G_0 =\emptyset$.

\begin{lemma}[Disjointness of domains and ranges]
\label{L disjoint domains and ranges}
For $k \in \N$ set
\begin{equation}\label{Iks}
I_k = \left\{
\begin{array}{ll}
\pi^M(G_{k}\cup G_{k+1}) & \mbox{\rm if $k$ odd and} \cr
\pi^N(G_{k} \cup G_{k+1}) & \mbox{\rm if $k$ even.}
\end{array}\right.
\end{equation}
Then the subsets $\{I_{2i+1}\}_{i=0}^\infty$ of $M$ are disjoint,  as are the subsets 
$\{I_{2i}\}_{i=1}^\infty$ of $N$.
\end{lemma}

\begin{proof}
For $i=\N$, we shall show the sets $I_{2i+1} \subset M$ are disjoint.  
Disjointness of the subsets $\{I_{2i}\}_{i=1}^{\infty}$ of $N$ is proved similarly, using Lemma
\ref{L-key}(b).

To derive a contradiction, suppose $x \in I_{2i+1} \cap I_{2j+1}$ with $i< j$.   
Depending on whether $i=0$ or $i\ge 1$,
there exist $(x,y) \in E_{1} \cup E_{2}^{h-} \cup E_{2i+1}^h \cup E_{2i+2}^{h-}$ and 
$(x,y') \in E_{2j+1}^h \cup E_{2j+2}^{h-}$.
Since $2i+2<2j+1$ disjointness of the $E_k$ imply $y \ne y'$.
Lemma \ref{L-key}(a) then asserts $(x,y') \in E_{2j+1}^{v-} \cup E_{2j+2}^{v-}$ --- the desired 
contradiction.  Thus the subsets $\{I_{2i+1}\}_{i=0}^{\infty}$ of $M$ are disjoint.
\end{proof}

\begin{lemma}[Numbered limbs]\label{L numbered limbs}
The Borel sets $\{G_{2i+1}\}_{i=0}^{\infty}$ of \eqref{Gis} form the graphs and 
$\{G_{2i}\}_{i=1}^{\infty}$ form the antigraphs of a numbered limb system:
$G_{2i+1} = \graph(f_{2i+1})$ and $G_{2i}=\antigraph(f_{2i})$,  with 
$\dom f_k \cup \ran f_{k+1} \subset I_k$ from \eqref{Iks} for all $k \in \{0,1,\ldots,K\}$.
\end{lemma}

\begin{proof}
The sets $G_k$ are Borel by their construction \eqref{Eks}, \eqref{Gis} and Lemma
\ref{LEM3}.
If $i>0$ we claim
$G_{2i+1} := E_{2i+2}^{h-} \cup E_{2i+1}^{hv} \cup E_{2i+1}^{h-}$ is a graph:
Let $(x,y) \ne (x,y')$ be distinct points in $G_{2i+1}$.  
Lemma \ref{L-key}(a) asserts that at least one of the 
two points lies in $E_{2i+1}^{v-}$ or $E_{2i+2}^{v-}$ --- a contradiction.
The fact that $G_{2i}$ is an antigraph follows by symmetry,  and the fact that $G_1$ 
is a graph is checked similarly.

We can therefore write $G_{2i+1} = \graph(f_{2i+1})$ and $G_{2i}=\antigraph(f_{2i})$
for some sequence of maps $f_k: \dom f_k \longrightarrow \ran f_k$ with domains
$\dom f_k \subset M$ and ranges $\ran f_k \subset N$ if $k$ odd,  and 
$\dom f_k \subset N$ and $\ran f_k \subset M$ if $k$ even.
The fact that $\dom f_k \cup \ran f_{k+1} \subset I_k$ follows directly from \eqref{Iks},
while
Lemma \ref{L disjoint domains and ranges} implies disjointness of the $I_{2i+1} \subset M$
and of the $I_{2i} \subset N$.
If
$\tilde M =M \setminus \cup_{i=0}^{\infty} I_{2i+1}$ or 
$\tilde N:= N \setminus \cup_{i=0}^{\infty} I_{2i}$ is non-empty,  we replace $I_0$ by $I_0 \cup \tilde N$
and $I_1$ by $I_1 \cup \tilde M$ to complete our verification of the properties of a numbered limb 
system (Definition \ref{DEFlimb}).
\end{proof}


\begin{proof}[Proof of Theorem \ref{THM1} and Remark \ref{R singular measures}]
To recapitulate:  Gangbo and McCann \cite{gm96} provide a $c$-compact set
${\mathcal S}$ containing the support of every optimizer $\gamma \in \Pi_0$,  and
a pair of Lipschitz potentials \eqref{pot1}--\eqref{potincl} such that ${\mathcal S} \subset \partial_c\psi$.
We take $\mathcal S$ to be the minimal such set without loss of generality.   Proposition \ref{PROP1}
shows $M_0:=M \setminus \dom d\psi$ to be $\mu$-negligible and $N_0:=N \setminus \dom d\phi$
to be $\nu$-negligible;  both are $\sigma$-compact by Proposition \ref{P Clarke}.
Without loss of generality, we therefore assume $M_0 \times N_0 \subset E_0$ and $M \times N \setminus \mathcal{S} \subset E_0$, 
the $\gamma$-negligible $\sigma$-compact set. Lemma \ref{L numbered limbs}
provides a decomposition \eqref{20may1} of $\tilde {\mathcal S} :=M \times N \setminus E_0$ into a numbered limb system consisting of Borel graphs and antigraphs --- apart from a Borel set $E_\infty = \bigcap \tilde {\mathcal S}_k$.
But we have $\gamma(E_\infty)=0$ for each $\gamma\in\Pi_0$ by hypothesis. Theorem~\ref{THMlimb} therefore asserts that
at most one $\gamma \in \Pi_0$ vanishes outside $\tilde {\mathcal S} \setminus E_\infty$.
But since all $\gamma \in \Pi_0$ have this property, $\Pi_0$ must be a singleton.  
Finally,  since $\tilde {\mathcal S}_{k+1} \subset \tilde {\mathcal S}_k$ we see 
$\pi^{M}(E_\infty) = \bigcap_{k=1}^\infty \pi^{M}(\tilde {\mathcal S}_k)$ and $\pi^N(E_\infty)$ are Borel using
Corollary \ref{CorProj}.
\end{proof}

\section{Proof of Theorem \ref{THMmotivation}}

Noting $\dim M=\dim N$, let $(\bar{x},\bar{y}) \in M \times N$ be such that $\frac{\partial^2 c}{\partial y\partial x} (\bar{x},\bar{y})$ is invertible.
The mapping 
$$
F \, : \, y \in N \, \longmapsto \, \frac{\partial c}{\partial x}(\bar{x},y) \in T_{\bar{x}}M
$$
is $C^1$ and since its differential at $\bar{y}$ is not singular, its image contains an open set in $T_{\bar{x}}M$. By Sard's theorem (see  \cite[\S 3.4.3]{Federer69}), the image of critical points of $F$ has Lebesgue
 measure zero, so we may assume without loss of generality that $F(\bar{y})$ is a regular value of $F$, meaning there is no $y$ with $\frac{\partial^2 c}{\partial y\partial x} (\bar{x},y)$ singular such that $F(y)=F(\bar{y})$.  The next lemma then follows from topological arguments.

\begin{lemma}[Generic failure of twist]
\label{LEMtopology}
Fix $(\bar x,\bar y) \in M \times N$ such that $F(\bar y)$ is a regular value of $F(y) =\frac{\partial c}{\partial x}(\bar x,y)$. There is $ \hat{y} \in N$ such that $F(\hat y) = F(\bar y)$, i.e.
\begin{eqnarray}\label{form1}
\frac{\partial c}{\partial x} \bigl(\bar{x},\bar{y}\bigr) = \frac{\partial c}{\partial x} \bigl(\bar{x},\hat{y}\bigr).
\end{eqnarray}
\end{lemma}

\begin{proof}[Proof of Lemma \ref{LEMtopology}]
 We argue by contradiction and assume that 
\begin{eqnarray*}
\forall y \in N, \quad y\neq \bar{y} \, \Longrightarrow \, F(y) \neq  F(\bar{y}).
\end{eqnarray*}
Note that since $F$ is a local diffeomorphism in a neighborhood of $\bar{y}$, the above condition still holds if we replace $F$ by $\tilde{F}$ a smooth (of class $C^{\infty}$) regularization of $F$ sufficiently close to $F$.  So without loss of generality we may assume that $F$ is smooth. Define the mapping $G: N \setminus \{\bar{y}\} \rightarrow \bfS^{n-1}$ by 
$$
G(y):= \frac{ F(y)-F(\bar{y})}{|F(y)-F(\bar{y})|} \qquad \forall y \in N \setminus \{\bar{y}\}.
$$ 
The mapping $G$ is smooth, so by Sard's Theorem it has a regular value $\lambda$. Then the set 
$$
G^{-1}(\lambda) := \Bigl\{y\in N \setminus \{\bar{y}\} \, \vert \, G(y) = \lambda \Bigr\}
$$
is a one dimensional submanifold of $N \setminus \{\bar{y}\}$. Moreover, since  the differential of $F$ at $\bar{y}$ is invertible, there are a open neighborhood $\mathcal U$ of $\bar{y}$ and  a $C^1$ curve $\gamma :[-\epsilon,\epsilon] \rightarrow N$ with $\gamma(0)=\bar{y}$ and $\dot{\gamma}(0)\neq 0$ such that
$$
G\left( \gamma(\pm t)\right) = \pm \lambda \qquad \forall t \in (0,\epsilon]
$$
and 
$$
G^{-1}(\lambda) \cap \mathcal{U} = \gamma \left((0,\epsilon)\right).
$$
This shows that the closure of $G^{-1}(\lambda)$ is a compact one dimensional submanifold whose boundary is $\bar{y}$. But the boundary of any compact submanifold of dimension one  is a finite set with even cardinal (see \cite{milnor97}), a contradiction.
\end{proof}

We need now to construct a $c$-convex function whose $c$-subdifferential at each point near $\bar x$ takes values near both $\bar{y}$ and $\hat{y}$. We  note that since $F(\bar{y})$ is a regular value of  $F(y) =\frac{\partial c}{\partial x}(\bar x,y)$ and $F(\hat y) = F(\bar y)$, both linear mappings $D_yF(\bar{y}), D_yF(\hat{y})$ are invertible.


\begin{lemma}\label{LEMpotentials}
There is a pair of functions $\psi : M \rightarrow \R, \, \phi : N \rightarrow \R$ such that 
\begin{eqnarray}\label{LEMpotentials1}
\psi(x) = \max_{y\in N} \Bigl\{ \phi(y) - c(x,y)\Bigr\} \qquad \forall x \in M
\end{eqnarray}
and
\begin{eqnarray}\label{LEMpotentials2}
\phi(y) = \min \Bigl\{ \psi(x)+c(x,y) \, \vert \, x \in M\Bigr\} \qquad \forall y \in N,
\end{eqnarray}
together with an open neighborhood $\bar{U}$ of $\bar{x}$, two open neighborhoods $\bar{V} \subset N$ of $\bar{y}$ and $\hat{V} \subset N$ of $\hat{y}$ with $\bar{V} \cap \hat{V}=\emptyset$, and two $C^1$ diffeomorphisms 
$$
\bar{y} : \bar{U} \rightarrow \bar{V} \quad \hat{y} : \bar{U} \rightarrow \hat{V} 
$$
with
$$
\bar{y} (\bar{x}) = \bar{y} \quad \mbox{and} \quad \hat{y} = \hat{y}(\bar{x}),
$$
such that
\begin{eqnarray}\label{LEMpotentials3}
\partial_c\psi(x) \cap \left( \bar{V} \cup \hat{V}\right)= \Bigl\{ \bar{y}(x), \hat{y}(x) \Bigr\} \qquad \forall x \in \bar{U}.
\end{eqnarray}
\end{lemma}

\begin{proof}[Proof of Lemma \ref{LEMpotentials}]
Since we work locally in neighborhoods of $\bar{x}$, $\bar{y}$ and $\hat{y}$, taking charts, we may assume that we work in $\R^n$. For every symmetric $n\times n$ matrix $Q$, there is a function $f:M \rightarrow \R$ of class $C^2$ such that 
\begin{eqnarray}\label{motiv1}
d_{\bar{x}}f  = - \frac{\partial c}{\partial x} \bigl(\bar{x},\bar{y}\bigr) = - \frac{\partial c}{\partial x} \bigl(\bar{x},\hat{y}\bigr)
\end{eqnarray}
and
\begin{eqnarray}\label{motiv2}
\mbox{Hess}_{\bar{x}} f = Q.
\end{eqnarray}
Let $Q$ be fixed such that
\begin{eqnarray}\label{motiv3}
 Q + \frac{\partial^2 c}{\partial x^2} \bigl( \bar{x},\bar{y}\bigr), \, Q + \frac{\partial^2 c}{\partial x^2} \bigl( \bar{x},\hat{y}\bigr) >0,
\end{eqnarray}
we claim that there is a $c$-convex function $\psi:M \rightarrow \R$ which coincide with $f$ in an neighborhood of $f$ and which satisfies the required properties. Since both $\frac{\partial^2 c}{\partial x \partial y} (\bar{x},\bar{y})$ and $\frac{\partial^2 c}{\partial x \partial y} (\bar{x},\hat{y})$ are invertible and (\ref{motiv1}) holds with $\bar{y}\neq \hat{y}$, the Implicit Function Theorem yields a open neighborhood $\bar{U} \subset M$ of $\bar{x}$, 
two disjoint open neighborhoods $\bar{V}, \hat{V} \subset N$ of $\bar{y}, \hat{y}$ respectively, and  two functions of class $C^1$
 $$
 x \in \bar{U} \longmapsto \bar{y}(x) \in \bar{V}, \quad  x \in \bar{U} \longmapsto \hat{y}(x) \in \hat{V}
 $$
such that
\begin{eqnarray}\label{motiv4}
\left\{
\begin{array}{rcl}
\bar{y}(\bar{x}) &=& \bar{y}\\
 \hat{y}(x) & = & \hat{y}
 \end{array}
 \right. \quad \mbox{and} \quad \left\{
\begin{array}{rcl} 
\frac{\partial c}{\partial x} (x,\bar{y}(x)) & = & - d_xf \\
\frac{\partial c}{\partial x} (x,\hat{y}(x)) & = & - d_xf
\end{array}
\right.
\quad \forall x \in \bar{U}.
\end{eqnarray}
Taking one derivative at $\bar{x}$ in the latter yields
$$
 \frac{\partial^2 c}{\partial x^2} \bigl(\bar{x},\bar{y}\bigr) +   \frac{\partial^2 c}{\partial y \partial x} \bigl(\bar{x},\bar{y}\bigr) \frac{\partial \bar{y}}{\partial x} (\bar{x}) =  \frac{\partial^2 c}{\partial x^2} \bigl(\bar{x},\hat{y}\bigr) +   \frac{\partial^2 c}{\partial y \partial x} \bigl(\bar{x},\hat{y}\bigr) \frac{\partial \hat{y}}{\partial x} (\bar{x}) = - \mbox{Hess}_{\bar{x}} f,
$$
which can be written as 
$$
\left\{
\begin{array}{rcl}
 \frac{\partial \bar{y}}{\partial x} (\bar{x}) & = & -  \left(  \frac{\partial^2 c}{\partial y \partial x} \bigl(\bar{x},\bar{y}\bigr)\right)^{-1} \left[ \mbox{Hess}_{\bar{x}} f +  \frac{\partial^2 c}{\partial x^2} \bigl(\bar{x},\bar{y}\bigr) \right]\\
  \frac{\partial \hat{y}}{\partial x} (\bar{x}) & = & -  \left(  \frac{\partial^2 c}{\partial y \partial x} \bigl(\bar{x},\hat{y}\bigr)\right)^{-1} \left[ \mbox{Hess}_{\bar{x}} f +  \frac{\partial^2 c}{\partial x^2} \bigl(\bar{x},\hat{y}\bigr) \right].
  \end{array}
  \right.
$$
Therefore, by (\ref{motiv2})-(\ref{motiv3})  we infer that $\frac{\partial \bar{y}}{\partial x} (\bar{x})$ and $\frac{\partial \hat{y}}{\partial x} (\bar{x})$ are invertible. Then restricting $\bar{U}, \bar{V}, \hat{V}$ if necessary, we may assume that the mappings 
$$
 x \in \bar{U} \longmapsto \bar{y}(x) \in \bar{V}, \quad  x \in \bar{U} \longmapsto \hat{y}(x) \in \hat{V}
$$
are diffeomorphisms. Moreover, the functions of class $C^2$ given by
$$
\bar{G} \, : \, x \in \bar{U} \, \longmapsto  \, f(x) - f(\bar{x}) + c \left(x,\bar{y}(x)\right) -  c \left(\bar{x},\bar{y}(x)\right)
$$
and 
$$
\hat{G} \, : \, x \in \bar{U} \, \longmapsto  \, f(x) - f(\bar{x}) + c \left(x,\hat{y}(x)\right) -  c \left(\bar{x},\hat{y}(x)\right)
$$
satisfy (using (\ref{motiv2})-(\ref{motiv4}))
$$
\bar{G} (\bar{x}) = \hat{G} (\bar{x}) =0, \quad d_{\bar{x}} \bar{G} = d_{\bar{x}} \hat{G}=0, \quad \mbox{Hess}_{\bar{x}} \bar{G} =  \mbox{Hess}_{\bar{x}} \hat{G}  = -  \left[ Q + \frac{\partial^2 c}{\partial x^2} \bigl( \bar{x},\bar{y}\bigr)\right] <0,
$$
so  we may also assume that 
\begin{eqnarray}\label{motiv5}
\left\{
\begin{array}{rcl}
     f(x') - f(x) + c \left(x',\bar{y}(x')\right) -  c \left(x,\bar{y}(x')\right) & <& 0  \\
     f(x') - f(x) + c \left(x',\hat{y}(x')\right) -  c \left(x,\hat{y}(x')\right) & <& 0  
       \end{array}
       \right.
        \qquad \forall x ' \in \bar{U} \setminus \{x\}, \, \forall x \in \bar{U}. 
\end{eqnarray}
As a matter of fact, freezing $x$ in the first line of (\ref{motiv5}) and setting 
$$
\bar{G}_x (x') =  f(x') - f(x) + c \left(x',\bar{y}(x')\right) -  c \left(x,\bar{y}(x')\right) \qquad \forall x' \in \bar{U},
$$ 
we check that for every $x\in \bar{U}$, we have
$$
\bar{G}_x (x)  =0, \quad d_x \bar{G}_x =d_xf +  \frac{\partial c}{\partial x} \bigl(x,\bar{y}(x)\bigr) =0
$$
and for every $x' \in \bar{U}$
\begin{eqnarray*}
d_{x'} \bar{G}_x =   d_{x'}f +  \frac{\partial c}{\partial x} \bigl(x',\bar{y}(x')\bigr) + \left[ \frac{\partial c}{\partial y} \bigl(x',\bar{y}(x')\bigr) -  \frac{\partial c}{\partial y} \bigl(x,\bar{y}(x')\bigr) \right]  \frac{\partial \bar{y}}{\partial x}(x') 
\end{eqnarray*}
which implies
\begin{eqnarray*}
\mbox{Hess}_{x} \bar{G}_x & = &  \mbox{Hess}_{x} f +  \frac{\partial^2 c}{\partial x^2} \bigl(x,\bar{y}(x)\bigr) +  \frac{\partial^2 c}{\partial y\partial x} \bigl(x,\bar{y}(x)\bigr) \frac{\partial \bar{y}}{\partial x}(x)  +  \frac{\partial^2 c}{\partial x \partial y} \bigl(x,\bar{y}(x)\bigr)  \frac{\partial \bar{y}}{\partial x}(x) 
\end{eqnarray*}

Define the functions $\phi_0 : N \rightarrow \R$ and $\psi : M \rightarrow \R$ by
$$
\phi_0(y) := \left\{ \begin{array}{ccl}
f\left( \bar{y}^{-1}(y)\right) + c\left( \bar{y}^{-1}(y),y\right) & \mbox{ if }  & y \in \bar{V}\\
f\left( \hat{y}^{-1}(y)\right) + c\left( \hat{y}^{-1}(y),y\right) & \mbox{ if }  & y \in \hat{V}\\
-\infty & \mbox{ if }  & y \notin \bar{V} \cup \hat{V}
\end{array}
\right. 
\qquad \forall y \in N
$$
 and 
 $$
\psi(x) = \max_{y\in N} \Bigl\{ \phi_0(y) - c(x,y)\Bigr\} \qquad \forall x \in M.
$$
We observe that we have for every $x\in  \bar{U}$,
$$
\phi_0\bigl(\bar{y}(x)\bigr) - c\bigl(x,\bar{y}(x)\bigr) = \phi_0\bigl(\hat{y}(x)\bigr) - c\bigl(x,\hat{y}(x)\bigr).
$$
Then we have 
 $$
\psi(x) = \max_{x'\in \bar{U}} \Bigl\{ \phi_0\bigl(\bar{y}(x')\bigr) - c\bigl(x',\bar{y}(x')\bigr)\Bigr\} = \max_{x'\in \bar{U}} \Bigl\{ \phi_0\bigl(\hat{y}(x')\bigr) - c\bigl(x',\hat{y}(x')\bigr)\Bigr\} \qquad \forall x \in M.
$$
By the above construction and (\ref{motiv5}), we have for every $x\in \bar{U}$ and any $x' \in \bar{U}\setminus \{x\}$,
\begin{eqnarray*}
\phi_0 \bigl(\bar{y}(x)\bigr) -  c\bigl(x,\bar{y}(x)\bigr) = f(x) > f(x') +  c \bigl(x',\bar{y}(x')\bigr) -  c \bigl(x,\bar{y}(x')\bigr) = \phi_0\bigl(\bar{y}(x')\bigr) -  c\bigl(x,\bar{y}(x')\bigr).
\end{eqnarray*} 
We infer that 
$$
\psi(x) = \phi_0\bigl(\bar{y}(x)\bigr) - c\bigl(x,\bar{y}(x)\bigr) = \phi_0\bigl(\hat{y}(x)\bigr) - c\bigl(x,\hat{y}(x)\bigr) = f(x) \qquad \forall x \in \bar{U}.
$$
Setting 
$$
\phi(y) = \min \Bigl\{ \psi(x)+c(x,y) \, \vert \, x \in M\Bigr\} \qquad \forall y \in N,
$$
we check that (\ref{LEMpotentials1})-(\ref{LEMpotentials3}) are satisfied.
$$
\psi(x) = \max_{y\in N} \Bigl\{ \phi(y) - c(x,y)\Bigr\} \qquad \forall x \in M.
$$
\end{proof}

Returning to the proof of the second case, let us consider an absolutely continuous probability measure $\mu$ on $M$ whose support is contained in $\bar{U}$. Then define the nonnegative measures $\bar{\nu}, \hat{\nu}$ on $N$ by
$$
\bar{\nu} := \frac{1}{2} \bar{y}_{\sharp}\mu \quad \mbox{ and } \quad \hat{\nu} := \frac{1}{2} \hat{y}_{\sharp}\mu,
$$
and set 
$$
\nu := \bar{\nu} + \hat{\nu}.
$$
Since the functions $\bar{y}$ and $\hat{y}$ are diffeomorphism, $\nu$ is an absolutely continuous probability measure on $N$ whose support in contained in $\bar{V}\cup \hat{V}$. Moreover, the plan $\bar{\gamma}$ defined by 
$$
\bar{\gamma} := \frac{1}{2} \left(Id,\bar{y}\right)_{\sharp} \mu +   \frac{1}{2} \left(Id,\hat{y}\right)_{\sharp} \mu
$$
has marginals $\mu$ and $\nu$ and is concentrated on the set of $(x,y)\in M\times N$ with $x\in \bar{U}$ and $y \in \partial_c\psi(x) \cap (\bar{V}\cup \hat{V})$. By (\ref{LEMpotentials1})-(\ref{LEMpotentials2}), any plan $\gamma$ with marginals  $\mu$ and $\nu$ satisfies
\begin{eqnarray*}
\int_{M\times N} c(x,y) \, d\gamma (x,y) & \geq & \int_{M \times N}\left[ \phi(y)-\psi(x)\right]  \, d\gamma (x,y) \\
& = & \int_N \phi(y) \, d\nu(y) - \int_M \psi(x) \, d\mu(x) \\
& = & \int_{\bar{V}\cup \hat{V}} \phi(y) \, d\nu(y) - \int_{\bar{U}} \psi(x) \, d\mu(x) \\
& = & \int_{M\times N} c(x,y) \, d\bar{\gamma} (x,y),
\end{eqnarray*}
with equality in the first inequality if and only if $\gamma$ is concentrated on the  set of $(x,y)\in M\times N$ with $x\in \bar{U}$ and $y \in \partial_c\psi(x) \cap (\bar{V}\cup \hat{V})$. This shows that $\bar{\gamma}$ is the unique optimal plan with marginals $\mu$ and $\nu$. 

It remains to show that the set of costs satisfying (\ref{ASSmotivation}) is open and dense in $C^2(M\times N;\R)$. The openness is obvious. Let us prove the density. Let $c$ be fixed in $C^2(M\times N;\R)$ such that  (\ref{ASSmotivation}) is not satisfied. Let $\bar{r}\in \{0,\ldots, n-1\}$ be the maximum of the rank of $\frac{\partial^2 c}{\partial y\partial x} (x,y)$ 
for $(x,y) \in M\times N$, pick some $(\bar{x},\bar{y}) \in M\times N$ such that 
$$
\mbox{rank} \left( \frac{\partial^2 c}{\partial y\partial x} (x,y)\right)=\bar{r}.
$$
Since the rank mapping is lower semicontinuous, there are two open sets $U\subset M$ and $V\subset N$ such that the rank of  $\frac{\partial^2 c}{\partial y\partial x} (x,y)$ is equal to $\bar{r}$ for any  $(\bar{x},\bar{y}) \in U\times V$. Moreover restricting $U$ and $V$ if necessary and taking local charts, we may assume that we work in $\R^n$. Let  $X : V \rightarrow \R^n$ be the mapping defined by $X(y)= \frac{\partial c}{\partial x} (\bar{x},y)$, for any $y\in V$. Doing a change of coordinates in $x$ and $y$ if necessary, we may assume 
that the $\bar{r}\times \bar{r}$ matrix 
$$
G = \left( \frac{\partial X_i  }{\partial y_j} (\bar{y})\right)_{1\leq i,j \leq \bar{r}}
$$
is invertible. Define the mapping $G :V \rightarrow \R^n$ by
$$
G (y_1, \ldots, y_n) = \left(X(y)_1, \ldots, X(y)_{\bar{r}}, y_{\bar{r}+1}, \ldots, y_n\right) \qquad \forall y \in V.
$$
The function $G$ is of class $C^1$  and by construction the differential of $G$ at $\bar{y}$ is invertible. Then $G$ is a local diffeomorphism from a open neighborhood $V' \subset V$ of $\bar{y}$ onto an open neighborhood $Z$ of $\bar{z}:=G(\bar{y})$. The function $X$ in $z$ coordinates is given by 
$$
\tilde{X}(z) := X_x \left( G^{-1}(z)\right) \qquad \forall z \in Z.
$$
By construction, we have
$$
\tilde{X}(z)_i = z_i \qquad \forall i =1, \ldots, \bar{r}, \, \forall z \in Z.
$$
Therefore, since $\tilde{X}$ has rank $\bar{r}$, the coordinates $\left( \tilde{X}_{\bar{r}+1}, \ldots, \tilde{X}_{n}\right)$ do not depend upon the variables $z_{\bar{r}+1},\ldots, z_n$. Let $\delta : \R^n \times \R^n\rightarrow \R$ be the smooth function defined by 
$$
\delta (x,z) = \sum_{i={\bar{r}+1}}^n x_i z_i \qquad \forall x,z \in \R^n
$$
and let $\varphi : \R^n \rightarrow [0,1]$ be a cut-off function which is equal to $1$ in a neighborhood of $G(\bar{y})$ and $0$ outside $Z$. Then for every $\epsilon >0$ the function 
$$
\tilde{c} \, : \, (x,z) \longmapsto \, c\left(x,G^{-1}(z)\right) + \epsilon \varphi (z) \delta (x,z)
$$
has a mixed partial derivative which is invertible at $(\bar{x},\bar{z})$ and tends to $c$ (in $(x,z)$ coordinates) in $C^2$ topology as $\epsilon>0$ goes to zero.

\section{Generic costs in smooth topology}\label{SECproofTHM2}

The proof of Theorem \ref{THM2} follows by classical transversality arguments. We refer the reader to \cite{gg73} for further details on the results from Thom transversality theory that we use below. 

Recall that $\dim M = \dim N=n$.
Denote by $J^2(M\times N;\R)$ the smooth manifold of $2$-jets from $M\times N$ to $\R$ and denote by $V$ the set consisting of $2$-jets $
\left( (x,y), \lambda, p, H\right) $ where $H$ is a 
symmetric matrix consisting of four $n \times n$ blocks 
$$
H = \left[ \begin{matrix} H_1 & H_2 \\ H_3 & H_4 \end{matrix} \right],
$$
with $H_2$ of corank $\geq 1$. The set $V$ is closed and stratified by the smooth submanifolds
$$
V_r := \Bigl\{ \left( (x,y), \lambda, p, H\right)  \, \vert \, \mbox{rank}(H_2)=r\Bigr\} \qquad \forall r=0, \ldots, n-1,
$$
of codimension $\geq 1$. By the Thom Transversality Theorem (see \cite[Theorem 4.9 p. 54]{gg73}), the set $\mathcal C_1$ of costs $c\in C^{\infty}(M\times N;\R)$ such that $j^2c(M\times N)$ is transverse to $V$ is residual. For these costs the set $\Sigma :=(j^2c)^{-1}(V) \subset M \times N$ is stratified of codimension $\geq 1$ and it is nonempty. 
As a matter of fact, for every $x\in M$, the mapping $\frac{\partial c}{\partial x} (x,\cdot) : N \rightarrow T_x^*M$ is smooth and its image $\mathcal{I}$ is a compact subset of $T_x^*M$. Thus for every boundary point $p\in \partial \mathcal{I}$, the function $\frac{\partial c}{\partial x} (x,\cdot)$ cannot be a local diffeomorphism in a neighborhood of any $y\in N$ such that $\frac{\partial c}{\partial x} (x,y)=p$, which shows that for such $y$ the linear mapping $\frac{\partial^2 c}{\partial y\partial x} (x,y)$ cannot be invertible. This shows that $\Sigma$ is not empty. The fact that $\Sigma$ is stratified of codimension $\geq 1$ (and so of zero measure) comes from the fact that it is the inverse image by $j^2c : M\times N \rightarrow J^2(M\times N; \R)$ of $V$ which is transverse to $j^2(M\times N)$  (see \cite[Theorem 4.4 p. 52]{gg73}).

Using a similar argument, we next show that the set of costs without periodic chains is residual in $ C^{\infty}(M\times N;\R)$.

\begin{lemma}[Cyclic chains yield optimal alternatives]
\label{LEMperiodic}
Let
$$
 \Bigl( (x_1,y_1), \ldots (x_L,y_L)\Bigr) \in \left( M\times N\right)^L
$$
be a 
chain with $x_2=x_1, x_L\neq x_1,$ and $y_L=y_1$. Then $L=2K$ for some integer $K\geq 2$ and
$$
\sum_{k=0}^{K-1} c\bigl( x_{2k+1},y_{2k+1}\bigr) = \sum_{k=0}^{K-1} c\bigl( x_{2k+1},y_{2k+2}\bigr). 
$$
\end{lemma}

\begin{proof}[Proof of Lemma \ref{LEMperiodic}]
We have for any $k\in \{0,K-1\}$,
$$
x_{2k+2} = x_{2k+1} \quad \mbox{ and } \quad y_{2k+3} = y_{2k+2}.
$$
Then, since the set $\left\{   (x_1,y_1), \ldots (x_L,y_L) \right\}$ is $c$-cyclically monotone, we have 
\begin{eqnarray*}
\sum_{k=0}^{K-1} c\bigl( x_{2k+1},y_{2k+1}\bigr)  & \leq & \sum_{k=0}^{K-1} c\bigl( x_{2k+1},y_{2k+3}\bigr) \\
& = &  \sum_{k=0}^{K-1} c\bigl( x_{2k+1},y_{2k+2}\bigr) \\
& = &  \sum_{k=0}^{K-1} c\bigl( x_{2k+2},y_{2k+2}\bigr) \\
& \leq &  c\bigl( x_{2},y_{2K}\bigr) +  \sum_{k=1}^{K-1} c\bigl( x_{2k+2},y_{2k}\bigr) \\
& = & \sum_{k=0}^{K-1} c\bigl( x_{2k+1},y_{2k+1}\bigr).
\end{eqnarray*}
We conclude easily.
\end{proof}
We need now to work with $1$-multijets of smooth functions from $M\times N$ to $\R$. For every even integer $L=2K\geq 4$, we denote by $W_L$ the set of tuples 
$$
\left( \Bigl( (x_1,y_1), \ldots (x_L,y_L)\Bigr), \Bigl(\lambda_1, \ldots, \lambda_L\Bigr),  \Bigl( (p_1^x,p_1^y), \ldots (p_L^x,p_L^y)\Bigr) \right)
$$
satisfying 
$$
(x_i,y_i) \neq (x_j,y_j) \qquad \forall i\neq j \in \{1,\ldots, L\},
$$
$$
\left\{
\begin{array}{rcl}
x_{2k+2} & = & x_{2k+1} \\ 
y_{2k+3} & = & y_{2k+2},
\end{array}
\right.
\qquad
\left\{
\begin{array}{rcl}
p_{2k+2}^x & = & p_{2k+1}^x \\ 
p_{2k+3}^y & = & p_{2k+2}^y,
\end{array}
\right.
$$
for all $k\in \{0,K-1\}$ and
$$
\sum_{k=0}^{K-1} \lambda_{2k+1} = \sum_{k=0}^{K-1} \lambda_{2k+2}. 
$$
The set $W_L$ is a submanifold of $J^1_L(M\times N;\R)$ of dimension
$$
D= 4Kn+L-1= (2n+1)L-1
$$
and $J^1_L(M\times N;\R)$ has dimension $(4n+1)L$. Thus $W_L$ has codimension $2nL+1$.


By the Multijet Transversality Theorem  (see \cite[Theorem 4.13 p. 57]{gg73}), for each $K=2,3\ldots$,  the set $\mathcal C_K$ of costs $c$ for which $j^1_{2K} c$ is transverse to $W_{2K}$ 
is residual.  The intersection 
$$
\mathcal C = \mathcal{C}_1 \cap \left(\cap_{K=2}^\infty \mathcal C_{K}\right) 
$$
satisfies the conclusions of Theorem \ref{THM2}.

\appendix
\section{Generic uniqueness of optimal plans for fixed marginals}\label{SECproofTHM3}

Elaborating on a celebrated result by Ma\~n\'e \cite{mane96} in the framework of Aubry-Mather theory, it is possible
to prove that for fixed marginals the set of costs for which uniqueness of optimal transport plans holds is generic. 
Such a result was first obtained by Levin \cite{Levin08}.  We include an argument here for comparison.

Let $M$ and $N$ be smooth closed manifolds (meaning compact, without boundary) of dimension $n\geq 1$, $c:M\times N \rightarrow \R$ be a cost function in $C^k(M\times N;\R)$ with $k\in \N \cup \{\infty\}$, and $\mu, \nu$ be two Borel probability mesures, we recall that $\Pi(\mu,\nu)$ denotes the set of probability measures in $M\times N$ with first and second marginals $\mu$ and $\nu$. By the way, a measure on $M\times N$ is a continuous linear 
functional on the set of continuous functions $C^0(M\times N;\R)$ and the set $E=C^0(M\times N;\R)^*$ of such measures is equipped with the topology of weak-$*$ convergence saying that some sequence $(\pi_l)_l$ in $E$  converges to $\pi\in E$ if and only if 
$$
\lim_{l \to \infty}
\int_{M \times N} f \, d\pi_l 
=\int_{M \times N} f \, d\pi,
$$
for every $f\in C^0(M\times N;\R)$. The following is classical.

\begin{lemma}\label{compactness}
The set $\Pi(\mu,\nu)$ is a nonempty compact convex set in $E$.
\end{lemma}

The following will also be useful. We refer the reader to \cite{gg73} for the definition of the $C^k$-topology.

\begin{lemma}\label{LEMuseful}
The mapping 
$$
(\pi,c) \in E \times C^k(M\times N;\R) \, \longmapsto \, \int_{M\times N} c \, d\pi
$$
is continuous with respect to the weak-$*$ topology on $E$ and the $C^k$-topology on $C^k(M\times N;\R)$. Moreover, for every $\pi_1, \pi_2 \in \Pi(\mu,\nu)$ with $\pi_1 \neq \pi_2$, there is 
$c\in C^k(M\times N;\R)$ such that
$$
\int_{M \times N} f \, d\pi_1 \neq \int_{M \times N} f \, d\pi_2.
$$  
\end{lemma}

For every $c\in C^k (M\times N;\R)$ let $\mathcal{M}(c)$ be the set of optimal transport plans between $\mu$ and $\nu$, that is
$$
\mathcal{M}(c) := \left\{ \pi \in \Pi(\mu,\nu) \, \vert \, \int_{M\times N} f \, d\pi \leq \int_{M\times N} f \, d\pi', \, \forall \pi' \in \Pi(\mu,\nu)\right\}.
$$
By construction, $\mathcal{M}(c)$ is a nonempty compact convex subset of $\Pi(\mu,\nu)$.

\begin{theorem}[Levin]
\label{THMmane}
There exists a residual set $\mathcal{C} \subset C^k(M\times N;\R)$ such that for every $c\in \mathcal{C}$, the set $\mathcal{M}(c)$ is a singleton.
\end{theorem}

Here {\em residual} refers to a countable intersection of sets with dense interiors.
Theorem \ref{THMmane} follows easily from results of Ma\~n\'e \cite{mane96} (or from arguments developed subsequently by Bernard and Contreras \cite{bc08}). For sake of completeness we provide its proof which is based (following the approach of Bernard and Contreras \cite{bc08}) on the next lemma.  It shows that near any given cost function can be found another for which the minimizing facet of $K$ has arbitrarily small diameter.

\begin{lemma}\label{LEMmane1}
The weak-$*$ topology on $K$ can be metrized by a distance $\tilde d$ with the following
property. Let $c_0 \in C^k(M\times N;\R)$ be fixed.  For every neighborhood $U$ of $c_0$ in $C^k(M\times N;\R)$ and every $\epsilon>0$, there is $c\in U$ such that 
$$
\mbox{diam} \bigl( \mathcal{M}(c)\bigr) < \epsilon.
$$ 
\end{lemma}

\begin{proof}[Proof of Lemma \ref{LEMmane1}]
Let $U$ and $\epsilon>0$ be fixed. By compactness of $K:=\Pi(\mu,\nu)$ with respect to the weak-$*$ topology, there is a sequence $\{f_l\}_{l\in \N}$ of continuous functions that defines a metric $\tilde{d}$ on $\Pi(\mu,\nu)$ by
$$
\tilde{d}(\pi_1,\pi_2) = \sum_{l=0}^{\infty} \frac{1}{2^l} \left| \int_{M\times N} f_l \, d\pi_1 -  \int_{M\times N} f_l \, d\pi_2 \right| \qquad \forall \pi_1, \pi_2 \in \Pi(\mu,\nu),
$$
which is compatible with the weak topology. We claim that there is an integer $\bar{l}>0$ and 
$$
c_1, \ldots, c_{\bar{l}} \in C^k(M\times N ;\R)
$$
such that the continuous map 
$$
P_{\bar l} \, : \,  \Pi(\mu,\nu) \longrightarrow \R^{\bar{l}}
$$
defined by 
$$
P_{\bar l} (\pi) := \left( \int_{M\times N} c_1 \, d\pi, \ldots, \int_{M\times N} c_{\bar{l}} \, d\pi \right) \qquad \forall \pi \in \Pi(\mu,\nu),
$$
satisfies 
\begin{eqnarray}\label{diameter}
\mbox{diam} \left(P_{\bar l}^{-1}(p)\right) <\epsilon \qquad \forall p\in \R^{\bar{l}}.
\end{eqnarray}
where the latter refers to the diameter with respect to $\tilde{d}$ of the set of measures in $\Pi(\mu,\nu)$ sent to $p$ through $P_{\bar l}$. For every $c\in C^k(M\times N;\R)$, let 
$$
W_c := \left\{ \bigl(\pi_1,\pi_2\bigr) \, \vert \, \int_{M \times N} c \, d\pi_1 \neq \int_{M \times N} c \, d\pi_2 \right\}.
$$ 
By Lemma \ref{LEMuseful}, the sets $W_c$ are open and their union covers the complement of the diagonal $D=\{(\pi,\pi) \, \vert \, \pi \in K\}$. Since this complement is open in the metrizable set $K \times K$, 
we can extract a countable subcovering from this covering. So there is a sequence $\{c_l\}_{l\in \N}$ such that 
\begin{eqnarray}\label{cover}
K \times K \setminus D = \bigcup_{l\in \N} W_{c_l}.
\end{eqnarray}
We need to check that $P_{\bar l}$ satisfies (\ref{diameter}) if $\bar l$ is large enough. If not, there are two sequences $\{\pi_l^1\}_l, \{\pi_l^2\}_l$ in $K$ such that
$$
P_l \bigl(\pi_l^1\bigr) = P_l \bigl(\pi_l^2\bigr) \quad \mbox{and} \quad \tilde{d} \bigl(\pi_l^1,\pi_l^2\bigr) \geq \epsilon \qquad \forall l.
$$
Then up to taking subsequences,  $\{\pi_l^1\}_l$ and $\{\pi_l^2\}_l$ converge respectively to some $\pi^1, \pi^2 \in K$ with $\tilde{d}(\pi^1,\pi^2)\geq \epsilon$.  But by (\ref{cover}), there is $m$ such that 
$$
\int_{M \times N} c_m \, d\pi^1 \neq \int_{M \times N} c_m \, d\pi^2.
$$
But we have 
$$
P_m \bigl(\pi_l^1\bigr) = P_m \bigl(\pi_l^2\bigr) \qquad \forall l \ge m,
$$
which passing to the limit gives $P_m \bigl(\pi^1\bigr) = P_m \bigl(\pi^2\bigr)$, a contradiction.

Let $K':=P_{\bar l}(K)$ which is a nonempty convex compact set in $\R^{\bar{l}}$, denote by  $\Phi : \R^{\bar{l}} \rightarrow \R$ the function defined by 
$$
\Psi (x) := \left\{ \begin{array}{ccl}
\min \left\{ \int_{M\times N} c_0 \, d\pi \, \vert \, \pi \in K \mbox{ s.t. } P_{\bar l}(\pi)=x \right\} & \mbox{ if } & x \in K' \\
+\infty & \mbox{ if } & x\notin K'
\end{array}
\right.
\qquad \forall x \in \R^{\bar{l}},
$$
and denote by $\Phi : \R^{\bar{l}} \rightarrow \R$ its conjugate, that is 
$$
\Phi (y)  :=  \sup_{x\in \R^{\bar{l}}} \Bigl\{ \langle y,x\rangle - \Psi(x) \Bigr\} = \max_{x\in K'} \Bigl\{ \langle y,x\rangle -\Psi(x) \Bigr\} =    \max_{\pi\in K} \left\{ \int_{M\times N} \sum_{l=1}^{\bar{l}} y_l  c_l \, d\pi -\Psi \bigl(P_{\bar l}(\pi) \bigr) \right\},
$$
for every  $y \in \R^{\bar l}$.  By construction, $\Phi$ is convex and finite on $\R^{\bar{l}}$, moreover for every $\bar{y} \in \R^{\bar{l}}$ and every $\bar{x} \in \R^{\bar{l}}$ such that $
 \Phi \bigl(\bar{y}\bigr)  =  \langle \bar{y},\bar{x}\rangle -\Psi \bigl( \bar{x} \bigr)$, we have
$$
\Phi\bigl( \bar{y} \bigr) + \langle y - \bar{y}, \bar{x} \rangle =  \langle y, \bar{x} \rangle -\Psi \bigl( \bar{x} \bigr) \leq  \Phi(y) \qquad \forall y\in \R^{\bar{l}}.
$$
This means that $\bar{x}$ belongs to $\partial \Phi\bigl(\bar{y}\bigr)$, the subdifferential of $\Phi$ at $\bar{y}$.
If in addition $\bar{\pi} \in K$  satisfies $P_{\bar l}\bigl( \bar{\pi}\bigr) = \bar{x}$ and
$$
\int_{M \times N}  \left( c_0  -   \sum_{l=1}^{\bar{l}} \bar{y}_l  c_l \right)    \, d\bar{\pi} \leq \int_{M \times N}  \left( c_0  -   \sum_{l=1}^{\bar{l}} \bar{y}_l  c_l \right)    \, d\pi \qquad \forall \pi \in K,
$$
then by definition of $\Psi$, we have
$$
 \int_{M\times N} \sum_{l=1}^{\bar{l}} \bar{y}_l  c_l \, d\pi -\Psi \bigl(P_{\bar l}(\pi) \bigr) \leq  \int_{M\times N} \sum_{l=1}^{\bar{l}} \bar{y}_l  c_l \, d\bar{\pi} -\Psi \bigl( \bar{x} \bigr), \qquad \forall \pi \in K.
$$
This means that 
$$
\mathcal{M} \left( c_0 -   \sum_{l=1}^{\bar{l}} \bar{y}_l  c_l \right) \subset P_{\bar l}^{-1} \left( \partial \Phi\bigl( \bar{y}\bigr) \right). 
$$
By Rademacher's theorem, for almost every $\bar{y} \in \R^{\bar{l}}$ the set $ \partial \Phi\bigl( \bar{y}\bigr)$ is a singleton. We conclude by (\ref{diameter}).
\end{proof}

Let us now prove Theorem \ref{THMmane}.

\begin{proof}[Proof of Theorem \ref{THMmane}]
For every integer $l>0$ denote by $\mathcal{C}_l$ the set of $c\in C^k(M\times N;\R)$ such that
$$
\mbox{diam} \bigl( \mathcal{M}(c)\bigr) < \frac{1}{l}.
$$ 
By the continuity part in Lemma \ref{LEMuseful}, each set $\mathcal{C}_l$ is open and by Lemma \ref{LEMmane1}, it is dense as well. Then the set 
$$
\mathcal{C} := \bigcap_{l\in \N^*} \mathcal{C}_l
$$
does the job.
\end{proof}

\

\end{document}